\documentclass[a4paper]{iopart}

\usepackage{hyperref}
\usepackage{graphicx}
\expandafter\let\csname equation*\endcsname\relax
\expandafter\let\csname endequation*\endcsname\relax
\usepackage{amsmath}
\usepackage{amsfonts}
\usepackage{amssymb} 
\usepackage{color}
\usepackage{subfig}
\usepackage{algorithm}
\usepackage{algpseudocode}
\usepackage{listings}
\usepackage{amsthm,nicefrac,array}
\usepackage{mathtools}
\usepackage[textfont=it,font=small]{caption}
\usepackage{xcolor}
\usepackage{enumerate}
\usepackage{thmtools}

\DeclareMathOperator*{\argmin}{\arg \min}
\DeclareMathOperator*{\argmax}{\arg \max}

\DeclareMathOperator{\uc}{\mathcal{U}}

\DeclareMathOperator{\vc}{\mathcal{V}}
\DeclareMathOperator{\hc}{\mathcal{H}}

\DeclareMathOperator{\lc}{\mathcal{L}}
\DeclareMathOperator{\rc}{\mathcal{R}}

\DeclareMathOperator{\dom}{dom}

\newcommand{\tv}{\text{TV}}

\newcommand{\norm}[1]{\left\| #1 \right\|}%
\newcommand{\R}{\mathbb{R}}%
\newcommand{\N}{\mathbb{N}}

\newcommand{\ul}{u_{\lambda}}

\newcommand{\ip}[2]{\left\langle #1 , #2 \right\rangle }

\theoremstyle{theorem}
\newtheorem{theorem}{Theorem}[section]
\newtheorem{lemma}[theorem]{Lemma}

\theoremstyle{definition}
\newtheorem{definition}[theorem]{Definition}

\theoremstyle{remark}
\newtheorem{remark}[theorem]{Remark}

\theoremstyle{proposition}
\newtheorem{proposition}[theorem]{Proposition}

\theoremstyle{definition}
\newtheorem{example}[theorem]{Example}

\theoremstyle{definition}
\newtheorem{assumption}[theorem]{Assumption}

\numberwithin{equation}{section}

\begin{document}

%
%
%
%
%
%
%

\title[Inverse Scale Space Decomposition]{Inverse Scale Space Decomposition}

\author{Marie Foged Schmidt$^1$, Martin Benning$^2$, Carola-Bibiane Sch\"{o}nlieb$^2$}

\address{$^1$ Technical University of Denmark, Department of Applied Mathematics and Computer Science, Asmussens All\'e, Bldg. 303B, 2800 Kgs. Lyngby, Denmark}
\address{$^2$ University of Cambridge, Department of Applied Mathematics and Theoretical Physics, Wilberforce Road, Cambridge CB3 0WA, United Kingdom}
\ead{mfsc@dtu.dk, mb941@cam.ac.uk, cbs31@cam.ac.uk}

\begin{abstract}
We investigate the inverse scale space flow as a decomposition method for decomposing data into generalised singular vectors. We show that the inverse scale space flow, based on convex and absolutely one-homogeneous regularisation functionals, can decompose data represented by the application of a forward operator to a linear combination of generalised singular vectors into its individual singular vectors. We verify that for this decomposition to hold true, two additional conditions on the singular vectors are sufficient: orthogonality in the data space and inclusion of partial sums of the subgradients of the singular vectors in the subdifferential of the regularisation functional at zero.

We also address the converse question of when the inverse scale space flow returns a generalised singular vector given that the initial data is arbitrary (and therefore not necessarily in the range of the forward operator). We prove that the inverse scale space flow is guaranteed to return a singular vector if the data satisfies a novel dual singular vector condition.

We conclude the paper with numerical results that validate the theoretical results and that demonstrate the importance of the additional conditions required to guarantee the decomposition result.


\end{abstract}

\submitto{\IP}

\maketitle

\section{Introduction}
Regularisation methods are essential tools for the stable approximation of solutions of ill-posed inverse problems. Hence, their analysis is a vital part of inverse problems research. For linear regularisation methods a quite complete theory based on singular value decomposition has been available for a while (cf. \cite{stewart1993early,engl1996regularization}). Only recently though a more general definition of singular vectors for non-linear regularisation methods has been formalised in \cite{Benning_Burger_2013}, despite many previous and recent works on generalised Eigenfunctions in the context of non-linear partial differential equations and functional inequalities (cf. \cite{agueh2010variational,berestycki1980existence,browder1965variational,dacorogna1992generalisation,dolbeault2012symmetry,franzina2010existence,henrot2006extremum,kawohl2000symmetry,kawohl2003isoperimetric,kawohl2007simplicity,weinstein1983nonlinear}), control theory (cf. \cite{fujimoto2004singular}), image processing (cf. \cite{meyer2001oscillating,alter2005characterization,alter2005evolution,caselles2008characterization,edelman1998geometry,chambolle2016geometric,chambolle2016total}), and machine learning (cf. \cite{buhler2009spectral,bresson2010total}). The generalised singular vector definition is based on the minimisation of convex, one-homogeneous but not necessarily differentiable functionals. Classical non-linear regularisation methods for which generalised singular vectors apply include the famous total variation (TV) regularisation model \cite{rudin1992nonlinear} and $\ell^1$-norm regularisation, which is a key tool in compressed sensing (cf. \cite{candes2005decoding,candes2006near,donoho2006compressed,donoho2003optimally}). Newer advancements include the total generalised variation (TGV, cf. \cite{bredies2010total,setzer2008variational,setzer2011infimal,bredies2011inverse,benning2013higher}), regularisations based on the $\ell^1$-norm in combination with the shearlet transform \cite{kutyniok2012shearlets}, vectorial total variation regularisation (cf. \cite{goldluecke2012natural,duran2015novel,duran2016collaborative}), structure-tensor extensions of the total variation (cf. \cite{brox2006nonlinear,grasmair2010anisotropic,lefkimmiatis2013convex,lefkimmiatis2015structure,estellers2015adaptive}), vector-field regularisation \cite{brinkmann2015regularization}, de-biased total variation-type regularisations \cite{brinkmann2016bias} and regularisations based on the $\ell^1$-norm in combination with tight wavelet frames \cite{cai2008framelet,cai2009split,cai2012image}, to name just a few. It has been shown that generalised singular vectors play a vital role in the understanding of those non-linear regularisation methods similar to the role of classical singular vectors for linear regularisations. Key results are that the non-linear regularisation methods are capable of recovering individual singular vectors from data given in terms of a forward model applied to a generalised singular vector and additional measurement errors, just with a systematic bias. Considering corresponding inverse scale space methods (cf. \cite{burger2006nonlinear}) as regularisation methods instead, these methods are capable of recovering generalised singular vectors without systematic bias (cf. \cite{Benning_Burger_2013}). The latter also has significant impact on discretisations of the inverse scale space method, such as Bregman iteration (cf. \cite{osher2005iterative,zhang2011unified}), linearised Bregman iterations (cf. \cite{yin2008bregman,cai2009linearized,cai2009convergence,yin2010analysis}) and modifications (cf. \cite{moeller2014color,zeune2016multiscale}).\\

In \cite{gilboa2013spectral} Gilboa initiated the idea of a non-linear spectral decomposition of singular vectors by defining a spectrum based on the forward-scale space formulation of the TV model. He was able to show that this non-linear spectrum can be converted into the original signal via a linear inverse transform. In \cite{gilboa2014total} he made the connection that a single generalised singular vector (called Eigenfunction therein) of TV can be isolated via this spectral TV transform. In \cite{burger2015spectral} the idea of the non-linear spectral transform has further been extended to arbitrary one-homogeneous functionals, and to variational regularisation as well as the inverse scale space method instead of just the forward TV flow.\\

The definition of the non-linear spectral transform immediately gave rise to the question of circumstances under which it would be possible to decompose a signal composed of singular vectors. In \cite{horesh2015multiscale} this question has briefly been addressed, hinting that some sort of orthogonality condition is necessary for perfect separation (see \cite[Proposition 1]{horesh2015multiscale}). These initial ideas have been made more precise in \cite{gilboa2015semi}; given two singular vectors that are (fully orthogonal \cite[Definition 5]{gilboa2015semi} and) linear in the subdifferential \cite[Definition 7]{gilboa2015semi}, the non-linear spectral transform can separate these singular vectors perfectly (see \cite[Proposition 4]{gilboa2015semi}). In \cite{burger2016spectral} it has further been shown that the  non-linear spectral transform can indeed decompose data into a finite set of singular vectors, given that the corresponding regularisation functional is a one-norm concatenated with a linear matrix such that the matrix applied to its transpose is diagonally dominant (see \cite[Theorem 9]{burger2016spectral}).\\

Despite containing impressive results, we however also observe limitations of the previously mentioned works. First of all, all research is restricted to the non-linear spectral transform, which (in its current form) is not applicable to general inverse problems. Secondly, we either have theoretical conditions that allow us to separate only two singular vectors for general one-homogeneous regularisation functionals, or we have conditions for a specific family of regularisers that allow the decomposition of data into multiple singular vectors. With this work we want to fill in the knowledge gaps by introducing two conditions that guarantee the decomposition of sums of finitely many singular vectors via the inverse scale space method. The first condition will be an orthogonality condition on the finite set of singular vectors, while the second condition will ensure that partial sums of their subgradients are contained in the subdifferential of the corresponding regularisation functional at zero. Our main contribution will be a theorem which states that perfect decomposition of a finite linear combination of generalised singular vectors is guaranteed if these two conditions are satisfied. Subsequently, we are also going to investigate the case of arbitrary data that is not given in terms of a finite linear combination of generalised singular vectors.\\

The paper is organised as follows. First, we are going to introduce notation and mathematical preliminaries that are necessary throughout this paper. Then we formulate and discuss the major contribution of this work which is a result that states that finite sums of generalised singular vectors can be decomposed via the inverse scale space method under suitable conditions. Subsequently we also address the question of how to characterise the first inverse scale space step in case the given data is arbitrary. In this case we investigate under which conditions the first step is a generalised singular vector. Finally, we will present numerical results to support the theoretical results and conclude with an outlook of open questions and problems.

\section{Mathematical Preliminaries}
\label{sec: Preliminaries}
In this section we set the notation, define assumptions and give mathematical preliminaries. Throughout the paper we consider inverse problems of the form
\begin{align}
Ku = f
\label{eq: IP}
\end{align}
where $K:\uc \to \hc$ is a bounded linear operator from a Banach space $\uc$ to a Hilbert space $\hc$. We will denote the space of bounded linear operators between two normed spaces $\uc$ and $\vc$ by $\lc(\uc,\vc)$. We denote by $\uc^*$ the dual Banach space of $\uc$ with norm
\begin{align*}
\norm{p}_{\uc^*} = \sup_{\norm{u}_{\uc} = 1} |p(u)| = \sup_{u\in\uc\backslash\{0\}} \frac{|p(u)|}{\norm{u}_{\uc}} = \sup_{\norm{u}_{\uc} \leq 1} |p(u)|,
\end{align*}
where functional $p(u) = \ip{p}{u}_{\uc^*\times\uc}$ is the dual pairing between $p$ and $u$ that we will abbreviate by $\ip{p}{u}$ throughout the course of this work. If $\uc$ is a Hilbert space then the dual pairing can be identified with the inner product on $\uc$.
\\

A popular variational framework for approximating solutions of (\ref{eq: IP}) is the Tikhonov-type variational regularisation framework of the form
\begin{align}
\hat{u}\in \argmin_{u\in\dom(J)} \left\{ \frac{1}{2}\norm{Ku-f}_{\hc}^{2} + \alpha J(u) \right\},
\label{eq: VariationalFramework}
\end{align}
where $J:\dom(J)\subseteq\uc\to\R\cup\{+\infty\}$ is a regularisation functional that incorporates prior information about $\hat{u}$ and $\alpha>0$ is a regularisation parameter that controls the impact of $J$ on $\hat{u}$. 
\\

Throughout this paper we assume that the regularisation functional $J$ is proper, lower semi-continuous (l.s.c) and convex. We will further assume that $J$ is absolutely one-homogeneous, i.e.
\begin{align}
J(cu) = |c|J(u),\quad\forall c\in\R.\label{eq:absonehom}
\end{align}
Such a functional is non-negative by definition and 
fulfils the triangle inequality 
\begin{align}
J(u+v)\leq J(u) + J(v).
\label{eq: TriangleIneqForJ}
\end{align}
Hence $J$ is in fact a semi-norm. Furthermore, if $v\in\dom(J)$ and $v_0\in\ker(J)$ then
\begin{align}
J(v+v_0) = J(v),
\label{eq: J_kernelinvariance}
\end{align}
see \cite[Lemma 3.2]{burger2016spectral}. The subdifferential $\partial J(u)$ of $J$ at some element $u\in\text{dom}(J)$ will be used several times throughout this work and is defined by
\begin{align*}
\partial J(u) = \left\{p\in\uc^*\,:\, J(v)-J(u)-\ip{p}{v-u} \geq 0, \,\forall v\in\text{dom}(J)\right\}.
\end{align*}
Since $J$ is absolutely one-homogeneous the subdifferential can be characterised by
\begin{align}
\partial J(u) := \{p\in\uc^* \,:\, \ip{p}{u} = J(u),\; \ip{p}{v} \leq J(v),\,\forall v\in\dom(J)\},
\label{eq: subdiff_characterization}
\end{align}
see for instance \cite[Lemma 3.12]{burger2013guide}. As a special case we have
\begin{align*}
\partial J(0) := \{p\in\uc^* \,:\, \ip{p}{v} \leq J(v),\,\forall v\in\dom(J)\}
\end{align*}
since $J(0) = 0$. We observe that the subdifferential for absolutely one-homogenous functionals is invariant with respect to positive scaling, i.e.
\begin{align}
\partial J(cu) = \partial J(u),\;\forall c>0.
\label{eq: subdiff_scaling_invariant}
\end{align}
We also want to emphasise that for a subgradient $p\in\partial J(u)$ we know by the characterisation (\ref{eq: subdiff_characterization}) that $\ip{p}{v} \leq J(v)$ for all $v\in\text{dom}(J)$. Now take $v\in\ker(J)$. Then $\ip{p}{v} \leq 0$. Using $-v$ we then get $\ip{p}{v} \geq 0$. This shows that for absolutely one-homogeneous functionals:
\begin{align}
p\in\partial J(u)\quad \Rightarrow \quad \ip{p}{v} = 0,\;  \forall v\in\ker(J).
\label{eq: SubgradientPerpToKernelJ}
\end{align}

In \cite{osher2005iterative} a contrast-enhancing alternative to \eqref{eq: VariationalFramework} named Bregman iteration has been introduced. This is an iterative regularisation method for which
\begin{align*}
u^{k+1}\in\argmin_{u\in\dom(J)} \left\{\frac{1}{2}\norm{Ku-f}_{\hc}^{2} + \alpha D_J^{p^k}(u, u^k) \right\}\, ,
\end{align*}
where $u^0\in\ker(J)$ and the subgradient $p^k\in\partial J(u^k)$ satisfies $p^0 \equiv 0$ and
\begin{align}
p^{k + 1} = p^k + \frac{1}{\alpha}K^*(f-Ku^{k + 1}),
\label{eq:p_update_BregmanIteration}
\end{align}
for all $k \in \N$. Here $K^*:\hc\to\uc^*$ denotes the Banach adjoint operator of $K$ defined by $\ip{Ku}{v}_{\hc} = \ip{u}{K^*v}_{\uc\times\uc^*}$ for $u\in\uc$, $v\in\hc$. The term $D_J^{p^k}(u^{k + 1}, u^k)$ represents the (generalised) Bregman distance between $u^{k + 1}$ and $u^k$ w.r.t. the functional $J$ and the subgradient $p^k \in \partial J(u^k)$. The Bregman distance (cf. \cite{bregman1967relaxation,burger2016bregman}) is defined by
\begin{align*}
D_J^p(u, v) = J(u) - J(v) - \ip{p}{u-v} \, \text{,} \qquad p \in \partial J(v) \, \text{.}
\end{align*}
Note that the generalised Bregman distance is always non-negative due to the convexity of $J$. Since the Bregman iteration has a semi-convergence behaviour a stopping criterium for the iteration is needed. \\

In the limit $\alpha\to\infty$ we can interpret $\Delta t = 1/\alpha$ as a time step that tends to zero and the update (\ref{eq:p_update_BregmanIteration}) of $p^{k+1}$ for the Bregman iteration can be interpreted as a forward Euler discretisation of the inverse scale space (ISS) flow (cf. \cite{burger2006nonlinear}) 
\begin{align}
\partial_t p(t) = K^*(f-Ku(t)),\quad p(t)\in\partial J(u(t)),
\label{eq: ISSflow}
\end{align}
with $p(0) = 0$ and $u(0) = u^0 \in\ker(J)$. If we assume $u(0) = 0\in\ker(J)$, from the definition of the ISS flow we then obtain $p(t) = tK^*f$ for $0\leq t < t_1$ for some $t_1 > 0$. Hence we would require $p(t) = tK^*f \in \partial J(0)$ for $0\leq t < t_1$, which by (\ref{eq: SubgradientPerpToKernelJ}) implies $\ip{K^*f}{v} = 0$ for all $v\in\ker(J)$. If this is not the case we can project $K^*f$ onto the space of elements $p\in\uc^*$ for which $\ip{p}{v} = 0$ for all $v\in\ker(J)$. In the end we can always add $u^0\in\ker(J)$ to $u(t)$. For the remainder of this work we therefore assume w.l.o.g. that $u(0)=u^0 = 0$ for the initial value of the ISS flow. \\ 

We want to point out that the solution $u(t)$ for the ISS flow has to satisfy an orthogonality condition:
\begin{proposition}[Necessary Orthogonality Condition]
\label{prop: general_orthogonality_cond}
Let $u(t)$ be a solution of the inverse scale space flow (\ref{eq: ISSflow}). Then $u(t)$ satisfies the orthogonality condition
\begin{align}
\ip{K^*(f-Ku(t))}{u(t)} = 0.
\label{eq: general_orthogonality_cond}
\end{align}
\end{proposition}
\begin{proof}
Using the chain rule for differentiation and $p(t)\in\partial J(u(t))$ we get
\begin{align*}
\frac{d }{d t}J(u(t)) = \ip{p(t)}{\partial_t u(t)}.
\end{align*}
On the other hand, since $p(t)\in\partial J(u(t))$ and by the absolute one-homogeneity of $J$ we get
\begin{align*}
\frac{d }{d t}J(u(t)) = \frac{d }{d t}\ip{p(t)}{u(t)} = \ip{p(t)}{\partial_t u(t)} + \ip{\partial_t p(t)}{u(t)}.
\end{align*}
Hence we see that a solution $u(t)$ of (\ref{eq: ISSflow}) needs to satisfy
\begin{align*}
\ip{\partial_t p(t)}{u(t)} = 0.
\end{align*}
From the definition of the inverse scale space flow we know $\partial_t p(t) = K^*(f-Ku(t))$ and hence
\begin{align*}
\ip{K^*(f - Ku(t))}{u(t)} = 0.
\end{align*}
\end{proof}

If we ignore the one-homogeneity of $J$ for a moment it becomes evident that the ISS flow is a generalisation of Showalter's method \cite{showalter1967representation}, i.e.
\begin{align}
\partial_t u(t) = K^*(f - Ku(t) ) \, \text{,}\label{eq:showalter}
\end{align}
which is \eqref{eq: ISSflow} for $J(u) = \frac{1}{2}\| u \|_{L^2(\Omega)}^2$, where $\Omega \subset \R^m$ is some domain. It is well known that for compact $K$ between Hilbert spaces, solutions of \eqref{eq:showalter} - as those of other linear regularisation methods - can conveniently be expressed in terms of $K$'s system of classical singular vectors. That is, for two orthonormal bases $\{ v_j \}_{j \in \N}$ of $\overline{\rc(K)}$ and $\{ u_j \}_{j \in \N}$ of $\overline{\rc(K^*)}$, where $\rc(K)$ is the range of $K$ and $\rc(K^*)$ is the range of $K^*$, and a null sequence $\{ \sigma_j \}_{ j \in \N}$ satisfying
\begin{align*}
Ku_j = \sigma_j v_j \qquad \text{and} \qquad K^* v_j = \sigma_j u_j
\end{align*}
for all $j \in \N$, we can solve \eqref{eq:showalter} via
\begin{align}
u(t) = \sum_{j = 1}^\infty \left( 1 - e^{- \sigma_j^2 t} \right) \frac{1}{\sigma_j} \langle f, v_j \rangle u_j \, \text{.}\label{eq:showaltersol}
\end{align}
If we now assume that our data is given in terms of some linear combination of finitely many singular vectors, i.e. we have $f = Ku^\dagger$ for
\begin{align*}
u^\dagger = \sum_{j = 1}^n \gamma_j u_j
\end{align*}
where $\gamma_j \in \R$ and $n \in \N$, then \eqref{eq:showaltersol} simplifies to 
\begin{align*}
u(t) = \sum_{j = 1}^n \left( 1 - e^{- \sigma_j^2 t} \right) \gamma_j u_j \, \text{.}
\end{align*}
Hence, Showalter's method allows us to recover a weighted linear combination of the singular vectors of $u^\dagger$ from data $f = Ku^\dagger$.\\

One of our main goals throughout the course of this work is to study solutions of the ISS flow \eqref{eq: ISSflow}. With the example of Showalter's method we have seen that for the specific choice of $J(u) = \frac{1}{2}\| u \|_{L^2(\Omega)}^2$ solutions of \eqref{eq:showalter} are just simple transformations of the classical singular vectors if the data $f$ is represented by a composition of those. However, as we are interested in \eqref{eq: ISSflow} in combination with absolutely one-homogeneous regularisation functionals, classical singular vectors are of no use for our study. A remedy to overcome this issue seems to be the use of so-called \emph{generalised singular vectors} as defined in \cite{Benning_Burger_2013}. Generalised singular vectors are based on the idea of minimising the generalised Rayleigh quotient
\begin{align}
u_g \in \argmin_{\substack{u \in \ker(J)^\perp \\ \|Ku\|_{\hc} = 1}} \frac{J(u)}{\| Ku \|_{\hc}} \label{eq:rayleighquot} \, \text{.}
\end{align}
Here $u_g$ denotes the so-called \emph{non-trivial ground state} of the functional $J$, following \cite{agueh2010variational}, for which we have $J(u_g) \leq J(u)$ for all $u \in \dom(J)$ with $\norm{Ku}_{\hc} = 1$. The set $\ker(J)^\perp$ is the orthogonal complement of $\ker(J)$ for which we follow the definition of \cite{Benning_Burger_2013} and define
\begin{align*}
\ker(J)^{\perp} = \{u\in\dom(J) \,:\, \ip{Ku}{Kv} = 0,\,\forall v\in\ker(J)\}.
\end{align*}
Note that we have to ensure $\ker(K) \cap \ker(J) = \{ 0 \}$ for this definition to make sense. Otherwise every element $u \in \ker(K) \cap \ker(J)$ would both satisfy $u \in \ker(J)$ and $u \in \ker(J)^\perp$. \\

We want to highlight that we require the normalisation constraint in \eqref{eq:rayleighquot}; otherwise every function $c u_g$ with $c\in\R\setminus\{0\}$ is a solution of \eqref{eq:rayleighquot} given that $u_g$ is already a solution of \eqref{eq:rayleighquot}. We could then make $c$ arbitrarily close to zero to decrease $J(c u_g) = |c| J(u_g)$ due to the absolute one-homogeneity of $J$. \\

Now let us reformulate \eqref{eq:rayleighquot} in terms of the Lagrange saddle-point problem
\begin{align}
(u_g, \mu_g, \eta_g) = \arg \min_u \max_\mu \max_\eta \left\{ \frac{J(u)}{\| Ku \|_{\hc}} + \mu \left( \| Ku \|_{\hc} - 1 \right) + \eta \langle Ku, Kv \rangle \right\} \, \text{,}\label{eq:saddlepointform}
\end{align}
with the corresponding Lagrange multipliers $\mu_g$ and $\eta_g$, for all $v \in \ker(J)$. We want to investigate the optimality condition of \eqref{eq:saddlepointform} with respect to $u_g$ more closely, i.e.
\begin{align}
0 \in \frac{\partial J(u_g)\|Ku_g\|_{\hc} - \frac{J(u_g)}{\|Ku_g\|_{\hc}} K^*Ku_g}{\|Ku_g\|_{\hc}^2} + \frac{\mu_g}{\|Ku_g\|_{\hc}} K^* Ku_g + \eta_g K^* Kv \, \text{.}\label{eq:sadpointoptu} 
\end{align}
Computing the dual pairing of \eqref{eq:sadpointoptu} with $u_g$ therefore yields the relation
\begin{align*}
0 = \underbrace{\frac{J(u_g)\|Ku_g\|_{\hc} - J(u_g)\|Ku_g\|_{\hc}}{\|Ku_g\|_{\hc}^2}}_{= 0} + \mu_g \underbrace{\|Ku_g\|_{\hc}}_{\neq 0} + \eta_g \underbrace{\langle Kv, Ku_g \rangle}_{= 0} \, \text{,}
\end{align*}
which implies $\mu_g = 0$. 
If we further compute the dual product of \eqref{eq:sadpointoptu} with $v \in \ker(J)$ we observe
\begin{align*}
0 = \underbrace{\frac{\langle p_g, v \rangle \|Ku_g\|_{\hc} - \frac{J(u_g)}{\|Ku_g\|_{\hc}} \langle Ku_g, Kv \rangle}{\|Ku_g\|_{\hc}^2}}_{= 0} + \frac{\mu_g}{\| Ku_g\|_{\hc}} \underbrace{\langle Ku_g, Kv \rangle }_{= 0} + \eta_g \| Kv \|_{\hc}^2  \, \text{,}
\end{align*}
where $\langle p_g, v \rangle = 0$ for $p_g \in \partial J(u_g)$ due to \eqref{eq: SubgradientPerpToKernelJ}. Hence, unless the kernel of $J$ is trivial, we immediately observe $\eta_g = 0$. \\ 

Investigating \eqref{eq:sadpointoptu} for $\mu_g = 0$, $\eta_g = 0$ and $\| Ku_g \|_{\hc} \neq 0$ leaves us with the condition
\begin{align}
\lambda_g K^* K u_g \in \partial J(u_g) \label{eq:gensingvec1} \, \text{,}
\end{align}
for $\lambda_g := J(u_g) / \|Ku_g\|_{\hc}^2$ and $p_g \in \partial J(u_g)$. Using the normalisation $\norm{Ku_g}_{\hc} = 1$ and due to the absolute one-homogeneity \eqref{eq:absonehom} we conclude
\begin{align*}
\lambda_g = \ip{p_g}{u_g} = J(u_g).
\end{align*}
Equation \eqref{eq:gensingvec1} now characterises the ground state that solves \eqref{eq:rayleighquot}. However, there may exist other functions satisfying 
\begin{align}
\lambda K^* K \ul &\in \partial J(\ul) \, \text{,}\label{eq:gensingvec}
\intertext{with}
\lambda = \frac{J(\ul)}{\norm{K\ul}_{\hc}^2} &\geq \lambda_g = \frac{J(u_g)}{\norm{Ku_g}_{\hc}^2} \, \text{.}\label{eq:gensingval}
\end{align}
A function $u_\lambda$ that satisfies \eqref{eq:gensingvec} and \eqref{eq:gensingval} is called a \emph{generalised singular vector} with \emph{generalised singular value} $\lambda$. We normally assume $\norm{K\ul}_{\hc} = \norm{Ku_g}_{\hc} = 1$. Throughout the course of this work we will use the term \textit{singular vectors} for generalised singular vectors as well. Note that the term \textit{vector} does not refer to an element in $\R^m$ but to an element in a possibly infinite dimensional space. Additionally, note that if $\ul$ is a singular vector then also $-\ul$ is a singular vector for absolutely one-homogeneous $J$.
\begin{example}
We want to briefly demonstrate that the concept of generalised singular vectors \eqref{eq:gensingvec} is indeed a generalisation of classical singular vectors in inverse problems. If we consider the absolutely one-homogeneous regularisation functional $J(u) = \| u \|_{L^2(\Omega)}$ for $\uc = L^2(\Omega)$ on some domain $\Omega \subset \R^m$ we obtain $\partial J(u) = \{ u / \| u \|_{L^2(\Omega)} \}$ for $u$ with $\| u \|_{L^2(\Omega)} \neq 0$. Hence, \eqref{eq:gensingvec} reads as
\begin{align*}
\lambda K^* K \ul = \frac{\ul}{\| \ul \|_{L^2(\Omega)}}
\end{align*}
for $\lambda = \| \ul \|_{L^2(\Omega)} / \| K\ul \|_{\hc}^2$. If we multiply this equation by $1/\lambda$ and define $\sigma = \| K\ul \|_{\hc} / \| \ul \|_{L^2(\Omega)}$, we therefore obtain the classical singular vector condition
\begin{align*}
K^* K \ul = \sigma^2 \ul \, \text{.}
\end{align*}
\end{example}
\begin{example}\label{exm:tvsingvec}
Let $K=I:H^1_0(\Omega) \rightarrow L^2_0(\Omega)$ be the embedding operator from $H^1_0(\Omega)$ into $L^2_0(\Omega)$, where subscript zero means that the funtions are zero on the boundary of the bounded domain $\Omega\subset \R^m$, and let $J(u) = \| \nabla u \|_{L^2(\Omega; \R^m)}$. Then \eqref{eq:gensingvec} reads as
\begin{align*}
- \frac{\| \nabla \ul \|_{L^2(\Omega; \R^m)}^2}{\| \ul \|_{L^2(\Omega)}^2} \ul = \Delta \ul \, \text{,}
\end{align*}
which is simply the Eigenvalue problem of the Laplace operator.
\end{example}
\begin{example}\label{exm:dictionaryexm}
Let $K = I : \ell^1(\R^m) \to \ell^2(\R^m)$ be an embedding operator, and let $J(u) = \norm{Wu}_{\ell^1(\R^m)}$ where $W:\ell^1(\R^m)\to\ell^1(\R^m)$ and $W^*W = I$. The subdifferential of $J$ at $u\in\ell^1(\R^m)$ is characterised by
\begin{align*}
\partial J(u) = \left( W^*\circ\partial\norm{\cdot}_{\ell^1(\R^m)}\right)(Wu) \text{,}
\end{align*}
where
\begin{align*}
\partial\norm{v}_{\ell^1(\R^m)} = \text{sign}(v), \quad \text{sign}(v)_i \in \left\{
        \begin{array}{ll}
            \{1\}, & \quad v_i>0\\
            \{-1\}, & \quad v_i<0 \\
            {[}-1,1{]}, & \quad v_i=0
        \end{array}
    \right. \, \text{.}
\end{align*}
Then every vector $u_{\lambda}\in\R^m$ making $Wu_{\lambda}$ consist of peaks of the same magnitude is a singular vector:
\begin{align*}
\lambda K^*Ku_{\lambda} = \frac{\norm{Wu_{\lambda}}_{\ell^1(\R^m)}}{\norm{u_{\lambda}}_{\ell^2(\R^m)}^2} u_{\lambda} = \frac{\norm{Wu_{\lambda}}_{\ell^1(\R^m)}}{\norm{u_{\lambda}}_{\ell^2(\R^m)}^2}W^*Wu_{\lambda} = W^*\left(\frac{\norm{Wu_{\lambda}}_{\ell^1(\R^m)}}{\norm{u_{\lambda}}_{\ell^2(\R^m)}^2}Wu_{\lambda}\right).
\end{align*}
We need to show $\frac{\norm{Wu_{\lambda}}_{\ell^1(\R^m)}}{\norm{u_{\lambda}}_{\ell^2(\R^m)}^2}Wu_{\lambda}\in\text{sign}(Wu_{\lambda})$. Now
\begin{align*}
\norm{u_{\lambda}}_{\ell^2(\R^m)}^{2} = \ip{u_{\lambda}}{u_{\lambda}} = \ip{u_{\lambda}}{W^*Wu_{\lambda}} = \ip{Wu_{\lambda}}{Wu_{\lambda}} = \norm{Wu}_{\ell^2(\R^m)}^{2}.
\end{align*}
If $c$ is the magnitude of the peaks of $Wu_{\lambda}$ and $n$ is the number of peaks then $\norm{Wu_{\lambda}}_{\ell^1(\R^m)} = n|c|$ and $\norm{Wu_{\lambda}}_{\ell^2(\R^m)}^{2} = nc^2$. Hence
\begin{align*}
\frac{\norm{Wu_{\lambda}}_{\ell^1(\R^m)}}{\norm{u_{\lambda}}_{\ell^2(\R^m)}^2}Wu_{\lambda} = \frac{1}{|c|}Wu_{\lambda}\in\text{sign}(Wu_{\lambda}).
\end{align*}
If we further want $u_{\lambda}$ to be normalised we need $c = \pm\frac{1}{\sqrt{n}}$. The singular value $\lambda$ then simplifies to $\lambda = \sqrt{n}$.
\end{example}
\begin{example}
Let $J(u) = \tv_*(u)$ be the slightly modified total variation regularisation as defined in \cite[Section 4.1]{Benning_Burger_2013}:
\begin{align*}
\tv_*(u) := \underset{\substack{\varphi\in C^{\infty}(\Omega;\R^m),\\ \norm{\varphi}_{L^{\infty}(\Omega,\R^m)}\leq 1}}{\sup} \int_{\Omega} u\,\text{div}\varphi\, dx,
\end{align*}
where $\Omega\subseteq\R^m$ is a bounded domain. In contrast to the normal $\tv$-functional we do not choose test functions $\varphi$ with compact support, which yields an additional boundary term. For functions $u\in W^{1,1}([0,1])\cap C([0,1])$ we have
\begin{align*}
\tv_*(u) = \int_{0}^{1} |u'(x)|\, dx + |u(1)| + |u(0)|.
\end{align*}
Let $K = I:BV([0,1])\to L^2([0,1])$ be an embedding operator. It is then shown in \cite[Theorem 2]{Benning_Burger_2013} that the Haar wavelet basis is an orthonormal set of singular vectors for $K$ and $J$. Examples are for instance the one-dimensional Haar wavelets
\begin{align*}
u_{\lambda_1}(x) = \left\{
        \begin{array}{ll}
            1, & 0\leq x < 1/2\\
            -1, & 1/2 \leq x < 1 \\
            0, & \text{otherwise}
        \end{array}
    \right. \, \text{,}\quad \text{and} \quad  u_{\lambda_2}(x) = \left\{
        \begin{array}{ll}
            \sqrt{2}, & 0\leq x < 1/4\\
            -\sqrt{2}, & 1/4 \leq x < 1/2 \\
            0, & \text{otherwise}
        \end{array}
    \right. \, \text{,}
\end{align*}
with $\lambda_1 = \tv_*(u_{\lambda_1}) = 4$ and $\lambda_2 = \tv_*(u_{\lambda_2}) = 4\sqrt{2}$.
\end{example}

To see why generalised singular vectors are interesting in the context of inverse scale space regularisation, we recall the following theorem from \cite{Benning_Burger_2013}:
\begin{theorem}[{\cite[Theorem 9]{Benning_Burger_2013}}]\label{thm:ReconOfOneSingVec}
Let $J$ be proper, convex, l.s.c and absolutely one-homogeneous and let $\ul$ be a generalised singular vector with corresponding singular value $\lambda$. Then, if the data $f$ is given by $f = \gamma K\ul$ for a positive constant $\gamma$, a solution of the inverse scale space flow \eqref{eq: ISSflow} is given by
\begin{align*}
u(t) = \begin{cases} 0, & t < t_1,\\ \gamma \ul, & t_1 \leq t, \end{cases}
\end{align*}
for $t_1 = \lambda/\gamma$.
\end{theorem}
Hence, the ISS flow \eqref{eq: ISSflow} is able to recover $\ul$ from data $f = K\ul$ with no systematic bias. In other words, generalised singular vectors are invariant under inverse scale space regularisation. This makes generalised singular vectors useful tools for the analysis of the ISS flow. \\

A second result that we want to recall is the following result in case of data $f = \gamma K\ul + g$ that can additively be decomposed into a singular vector and an arbitrary element $g \in \hc$.
\begin{theorem}[{\cite[Theorem 10]{Benning_Burger_2013}}]\label{thm:noisyiss}
Let $J$ be proper, convex, l.s.c and absolutely one-homogeneous and let $\ul$ be a generalised singular vector with corresponding singular value $\lambda$. Then, if the data $f$ is given by $f = \gamma K\ul + g$ for a positive constant $\gamma$, such that there exist constants $\mu$ and $\eta$ satisfying
\begin{align}
\mu K^* Ku_{\lambda} + \eta K^* g \in \partial J(u_\lambda)\label{eq:noisecond}
\end{align}
with $\gamma > \mu/\eta$, a solution of the inverse scale space flow \eqref{eq: ISSflow} is given by
\begin{align*}
u(t) = \begin{cases} 0, & t < t_1,\\ \left( \gamma + \frac{\lambda - \mu}{\eta} \right) \ul, & t_1 \leq t < t_2, \end{cases}
\end{align*}
for $t_1 = (\lambda \eta)/(\lambda + \gamma\eta - \mu) < t_2 = \eta$.
\end{theorem}
The original rationale behind Theorem \ref{thm:noisyiss} was to allow for deterministic noise in the data, whilst showing that the singular vector can still be recovered. However, the additional term $g$ can obviously also be interpreted as a correction term for the singular vector $u_\lambda$ to match data $f$. \\

In fact, \eqref{eq:noisecond} is nothing but a specific type of source condition. For a given function $u^\dagger$, the \emph{source condition} is defined as
\begin{align*}
\rc(K^*) \cap \partial J(u^\dagger) \neq \emptyset \, \text{.}
\end{align*}
It ensures the existence of a source element $v \in \hc \setminus \{ 0 \}$ such that
\begin{align}
K^* v \in \partial J(u^\dagger) \tag{SC}\label{eq:sc}
\end{align}
is satisfied. We immediately see that \eqref{eq:noisecond} is equivalent to \eqref{eq:sc} for the specific choice $v = \mu K\ul + \eta g$ and $u^\dagger = \ul$. \\

The (generalised) \emph{strong source condition} as introduced in \cite{resmerita2005regularization}, in contrast to \eqref{eq:sc}, is defined as
\begin{align*}
\rc(K^* K) \cap \partial J(u^\dagger) \neq \emptyset \, \text{,}
\end{align*}
which guarantees the existence of a source element $w \in \uc \setminus \{ 0 \}$ with
\begin{align}
K^* Kw \in \partial J(u^\dagger) \, \text{.} \tag{SSC}\label{eq:ssc}
\end{align}
We want to highlight the following connections between the singular vector condition \eqref{eq:gensingvec} and the two source conditions \eqref{eq:sc} and \eqref{eq:ssc}. Starting with \eqref{eq:ssc} it is obvious that any singular vector $\ul$ satisfying \eqref{eq:gensingvec} does also satisfy \eqref{eq:ssc} with $u^\dagger = \ul$ and $w = \lambda \ul$. The converse is not true in general, which we want to demonstrate with the following example. Let $\uc = \hc = L^2(\Omega)$, $J(u) = \| u \|_{L^2(\Omega)}$ and $u^\dagger\in L^2(\Omega)$ with $\|u^\dagger \|_{L^2(\Omega)} = 1$. Then \eqref{eq:ssc} reads as $K^* K w = u^\dagger$ and we see immediately that $w$ does not need to be a multiple of $u^\dagger$ in this case. We now want to look more closely into the connection between \eqref{eq:gensingvec} and \eqref{eq:sc}.
\begin{proposition} 
\label{prop:svecscconnection}
Let $A \in \lc(\uc, \hc)$ and $f \in \hc$, and define $K \in \lc(\uc, \R)$ with $Ku := \langle A^* f, u \rangle$ for all $u \in \uc$. Then the singular vector condition \eqref{eq:gensingvec} for $K$ is equivalent to the source condition \eqref{eq:sc}
\begin{align*}
A^* v_\lambda \in \partial J(\ul) \, \text{,}
\end{align*}
for $v_\lambda = J(\ul) \langle A^* f, \ul \rangle f$ and $\left|\ip{A^*f}{u_{\lambda}}\right| = 1.$
\end{proposition}
\begin{proof}
With the adjoint operator $K^*$ of $K$ being given as
\begin{align*}
K^* c = c A^* f \, \text{,}
\end{align*}
for all $c \in \R$, the singular vector condition \eqref{eq:gensingvec} reads as
\begin{align*}
\lambda \langle A^* f, \ul \rangle A^* f \in \partial J(\ul) \, \text{,}
\end{align*}
for $| \langle A^* f, \ul \rangle | = 1$. Due to \eqref{eq:gensingval} we further conclude $\lambda = J(\ul)$, which verifies the proposition.
\end{proof}

\begin{remark}
Note that we either have $\langle A^* f, \ul \rangle = 1$ or $\langle A^* f, \ul \rangle = -1$, due to the normalisation constraint $| \langle A^* f, \ul \rangle | = 1$. Without loss of generality we assume $\langle A^* f, \ul \rangle = 1$; in this case, the element $v_\lambda$ in the source condition simplifies to $v_\lambda = J(\ul) f$. \\
\end{remark}

We summarise the assumptions that will be used throughout this work:
\begin{assumption}[Setup] 
\label{assumption: overall}
\quad
\begin{itemize}
\item $\uc$ is a Banach space.
\item $\hc$ is a Hilbert space.
\item $K\in\lc(\uc,\hc)$.
\item $J:\text{dom}(J)\subseteq\uc\to\R\cup\{+\infty\}$ is a proper, lower semi-continuous, convex, and absolutely one-homogeneous functional.
\item $\ker(K) \cap \ker(J) = \{ 0 \}$.
\item $f\in\hc$ and $\ip{K^*f}{v} = 0$ for all $v\in\ker(J)$. Then we can assume $u(0) = u^0 = 0$ for the initial value of the ISS flow.
\end{itemize}
\end{assumption}

\section{Decomposition of Generalised Singular Vectors}\label{sec:decomp}
Throughout this section we use Assumption \ref{assumption: overall}. In addition we assume that the input data $f$ for the inverse scale space flow (\ref{eq: ISSflow}) is represented by a linear combination of generalised singular vectors, i.e. $f = \sum_{j=1}^{n} \gamma_jKu_{\lambda_j}$ where $\{u_{\lambda_j}\}_{j=1}^n$ is a finite set of generalised singular vectors for $J$ and $K$ and $\gamma_j\in\R$, $j=1,\dots,n$. The main goal is to investigate if and under which conditions the ISS flow will give a perfect decomposition into the singular vectors representing $f$, i.e. when do we obtain
\begin{align}
u(t) = \left\{
        \begin{array}{ll}
            0, & \quad 0\leq t < t_1, \\[0.1cm]
            \sum_{j=1}^{k} \gamma_j u_{\lambda_j}, & \quad t_k \leq t < t_{k+1} \,\text{ for }\, k=1,...,n-1, \\[0.2cm]
            \sum_{j=1}^{n} \gamma_j u_{\lambda_j}, & \quad t_n \leq t,
        \end{array}
    \right.
\label{eq: DecompSol}
\end{align}
for fixed $t_k > 0$, $k=1,...,n$. We will introduce two conditions that will be imposed on the singular vectors in order for the ISS flow to behave this way. In addition we discuss the possibility of singular vector fusion, i.e. the possibility that two or more singular vectors can add up to another singular vector, which would make a decomposition impossible due to Theorem \ref{thm:ReconOfOneSingVec}. Our results will be compared to the achievements of \cite{gilboa2015semi} which deals with signal denoising and data represented by two singular vectors. 

\subsection{Orthogonality Condition}
We give a sufficient orthogonality condition on the generalised singular vectors $\{u_{\lambda_j}\}_{j=1}^n$ representing the data $f = \sum_{j=1}^{n} \gamma_jKu_{\lambda_j}$, for the ISS flow to satisfy the necessary orthogonality condition (\ref{eq: general_orthogonality_cond}). \\

Let the data for the inverse scale space flow be given as $f = \sum_{j=1}^{n} \gamma_j K u_{\lambda_j}$ where $\{u_{\lambda_j}\}_{j=1}^n$ is a set of generalised singular vectors. Assume that the inverse scale space flow gives a decomposition into the singular vectors of $f$ with no loss of contrast, i.e. the solution of the ISS flow is given by (\ref{eq: DecompSol}). The necessary orthogonality condition (\ref{eq: general_orthogonality_cond}) then states that
\begin{align*}
0 &= \ip{K^*\left(\sum_{j=1}^{n}\gamma_j K u_{\lambda_j} - \sum_{j=1}^{k} \gamma_j K u_{\lambda_j}\right)}{\sum_{j=1}^{k} \gamma_j u_{\lambda_j}} \\
&= \ip{\sum_{j=k+1}^{n} \gamma_j K u_{\lambda_j}}{\sum_{j=1}^{k} \gamma_j Ku_{\lambda_j}}
\end{align*}
for $k = 1,2,\dots,n-1$. A sufficient condition for this to hold is $K$-orthogonality between the singular vectors:
\begin{align*}
\ip{Ku_{\lambda_j}}{Ku_{\lambda_k}} = 0,\; j\neq k. \tag{\textbf{OC}}
\label{eq: orthogonality_cond}
\end{align*}
This orthogonality condition will be used throughout the entire paper.

\subsection{\eqref{eq: subdiff_cond} Condition}
$K$-orthogonality between the singular vectors representing $f$ will not be enough to guarantee that the solution of the ISS flow is given by (\ref{eq: DecompSol}). We therefore introduce an additional condition on the subgradients of the singular vectors that we name the (\ref{eq: subdiff_cond}) condition. It is a condition on the sum of the subgradients of a set of singular vectors which turns out to be a kind of linearity condition for $\partial J$ in the singular vectors. The (\ref{eq: subdiff_cond}) condition furthermore implies linearity of $J$ in the singular vectors. \\
\begin{definition}[\ref{eq: subdiff_cond} Condition]
\label{def: subdiff_cond}
Let $\{u_{\lambda_j}\}_{j=1}^{n}$ be a set of singular vectors of $J$ with corresponding singular values $\{\lambda_j\}_{j=1}^{n}$. We say that $\{u_{\lambda_j}\}_{j=1}^{n}$ satisfy the \textit{(\ref{eq: subdiff_cond}) condition} if
\begin{align}
\sum_{j=1}^{k} \lambda_jK^*Ku_{\lambda_j} \in \partial J(0),\;\forall\,k\in\{1;n\},\tag{\textbf{SUB0}}
\label{eq: subdiff_cond}
\end{align}
where $k\in\{1;n\}$ means $k=1,2,\dots,n$.
\end{definition}

\begin{remark}
In the definition of the (\ref{eq: subdiff_cond}) condition we use the notation $k\in\{1;n\}$ which means $k=1,2,\dots,n$. We will use this notation throughout the paper. We want to emphasise that in order for a set of singular vectors to satisfy the (\ref{eq: subdiff_cond}) condition it is necessary that the inclusion is satisfied for all $k\in\{1;n\}$. It is not enough that the full sum over the subgradients is included in the subdifferential at zero. We will call the inclusions for $k\in\{1;n-1\}$ the \emph{partial sum conditions} of the (\ref{eq: subdiff_cond}) condition. \\
\end{remark}

The (\ref{eq: subdiff_cond}) condition is equivalent to a linearity condition for the subdifferential for $K$-orthogonal singular vectors:
\begin{proposition} 
\label{prop: subdifferential_linearity_equivalences}
Let $\{u_{\lambda_j}\}_{j=1}^{n}$ be a set of $K$-normalised singular vectors of $J$ with corresponding singular values $\{\lambda_j\}_{j=1}^{n}$. Assume that they satisfy the orthogonality condition (\ref{eq: orthogonality_cond}). Then the (\ref{eq: subdiff_cond}) condition is satisfied if and only if
\begin{align*}
\sum_{j=1}^{k} \lambda_jK^*Ku_{\lambda_j} \in \partial J\left(\sum_{j=1}^{k} c_ju_{\lambda_j}\right),\; c_j\geq 0, \; \forall\,k\in\{1;n\}.
\end{align*}
\end{proposition}
\begin{proof}
Proof of "$\Leftarrow$": This is straightforward using the subdifferential characterisation (\ref{eq: subdiff_characterization}). \\

Proof of "$\Rightarrow$": Using the (\ref{eq: subdiff_cond}) condition, the triangle inequality (\ref{eq: TriangleIneqForJ}) for $J$ and the orthogonality condition (\ref{eq: orthogonality_cond}) we get
\begin{align*}
J(v) & - J\left(\sum_{j=1}^{k} c_ju_{\lambda_j}\right) - \ip{\sum_{j=1}^{k} \lambda_jK^*Ku_{\lambda_j}}{v-\sum_{j=1}^{k} c_j u_{\lambda_j}} \\
&\geq J(v) - \sum_{j=1}^{k} c_j J(u_{\lambda_j}) - J(v) + \sum_{j=1}^{k} c_j J(u_{\lambda_j}) \\
&= 0
\end{align*}
for $c_j\geq 0$, $j = 1,...,k$, for all $v\in\dom(J)$. This shows that $\sum_{j=1}^{k} \lambda_jK^*K u_{\lambda_j} \in\partial J\left(\sum_{j=1}^{k} c_j u_{\lambda_j}\right)$ for all $k\in\{1;n\}$.
\end{proof}
We would like to make a comparison to the (LIS) condition in \cite[Definition 7]{gilboa2015semi}. LIS stands for linearity in the subdifferential. The (LIS) condition is in principle the same as the (\ref{eq: subdiff_cond}) condition through Proposition \ref{prop: subdifferential_linearity_equivalences}, except that the (LIS) condition is only defined for a pair of two singular vectors and additionally requires the linearity in the subdifferential to hold for negative coefficients $c_j<0$ as well. (LIS) turns out to imply directly orthogonality between the singular vectors, see \cite[Proposition 4]{gilboa2015semi}. \\ 

The (\ref{eq: subdiff_cond}) condition furthermore implies some kind of linearity of $J$ for $K$-orthogonal singular vectors:
\begin{proposition} 
\label{prop: J_linearity}
Let $\{u_{\lambda_j}\}_{j=1}^{n}$ be a set of $K$-normalised singular vectors of $J$ with corresponding singular values $\{\lambda_j\}_{j=1}^{n}$. Assume that they satisfy the orthogonality condition (\ref{eq: orthogonality_cond}) and the (\ref{eq: subdiff_cond})  condition. Then
\begin{align*}
J\left(\sum_{j=1}^{n} c_ju_{\lambda_j}\right) = \sum_{j=1}^{n} c_jJ(u_{\lambda_j})
\end{align*}
for $c_j\geq 0$, $j = 1,...,n$.
\end{proposition}
\begin{proof}
Using Proposition \ref{prop: subdifferential_linearity_equivalences}, the properties of the subdifferential for absolutely one-homogeneous functionals, and (\ref{eq: orthogonality_cond}) we get
\begin{align*}
J\left( \sum_{j=1}^{n} c_ju_{\lambda_j} \right) = \ip{\sum_{j=1}^{n} \lambda_jK^*Ku_{\lambda_j}}{\sum_{j=1}^{n} c_j u_{\lambda_j}} = \sum_{j=1}^{n} c_j J(u_{\lambda_j}).
\end{align*}
\end{proof}
\begin{remark}
Note that the linearity in $J$ is only satisfied for non-negative coefficients $c_j$ and the proposition actually holds even if (\ref{eq: subdiff_cond}) is only satisfied for $k=n$.\\
\end{remark}
\begin{remark}
Linearity of $J$ does not imply the (\ref{eq: subdiff_cond}) condition under (\ref{eq: orthogonality_cond}), see Example \ref{example: l1_convolution} for a counter example. \\
\end{remark}

In order to give an intuitive idea of what the orthogonality condition (\ref{eq: orthogonality_cond}) and the (\ref{eq: subdiff_cond}) condition require of the singular vectors, we give three examples:

\begin{example}
Let us again consider $K = I : \ell^1(\R^m) \to \ell^2(\R^m)$ and $J(u) = \| Wu \|_{\ell^1(\R^m)}$ as in Example \ref{exm:dictionaryexm}. In order for two singular vectors $u_{\lambda_1}$ and $u_{\lambda_2}$ to be orthogonal we need
\begin{align*}
0 = \ip{u_{\lambda_1}}{u_{\lambda_2}} = \ip{u_{\lambda_1}}{W^*Wu_{\lambda_2}} = \ip{Wu_{\lambda_1}}{Wu_{\lambda_2}}.
\end{align*}
Hence we can just check that $Wu_{\lambda_1}$ and $Wu_{\lambda_2}$ are orthogonal in $\ell^2(\R^m)$. For the (\ref{eq: subdiff_cond}) condition we have
\begin{align*}
\lambda_1u_{\lambda_1} + \lambda_2u_{\lambda_2} &= \lambda_1W^*Wu_{\lambda_1} + \lambda_2W^*Wu_{\lambda_2} \\
&= W^*(\text{sign}(Wu_{\lambda_1})+\text{sign}(Wu_{\lambda_2})).
\end{align*}
If $Wu_{\lambda_1}$ and $Wu_{\lambda_2}$ have non-overlapping supports then
\begin{align*}
\lambda_1u_{\lambda_1} + \lambda_2u_{\lambda_2} = W^*(\text{sign}(Wu_{\lambda_1}+Wu_{\lambda_2}))\in\partial J(0)
\end{align*}
and the (\ref{eq: subdiff_cond}) condition is satisfied. Note that in this case also the orthogonality condition (\ref{eq: orthogonality_cond}) is satisfied.
\end{example}
\begin{example}
\label{example: l1_convolution}
Let $J(u) = \norm{u}_{\ell^1(\R^m)}$ with $\partial J(u) = \text{sign}(u)$, see Example \ref{exm:dictionaryexm}. We look at a convolution problem where $K:\ell^1(\R^m)\supseteq\ell^2(\R^m)\to\ell^2(\R^m)$ is the convolution operator defined by
\begin{align*}
(Ku)_k = \sum_{j=-1}^{1} u_{k-j}g_j, \quad g = \frac{1}{\sqrt{2}}\left( \begin{array}{c}
1 \\
1 \\
0 \end{array} \right).
\end{align*}
The adjoint operator $K^*$ is determined by
\begin{align*}
(K^*v)_j = \sum_{k=-1}^{1} v_{j+k}g_k.
\end{align*}
Two examples of $K$-normalised singular vectors are
\begin{align*}
u_{\lambda_1} = (0,0,1,0,0)^T \qquad \text{and} \qquad u_{\lambda_2} = \frac{1}{\sqrt{2}}(0,1,0,-1,0)^T.
\end{align*}
We have
\begin{alignat*}{2}
Ku_{\lambda_1} & = (0,1/\sqrt{2},1/\sqrt{2},0,0)^T, &&\quad \norm{Ku_{\lambda_1}}_{\ell^1(\R^m)} = 1, \\ \quad Ku_{\lambda_2} & = (1/2,1/2,-1/2,-1/2,0)^T, &&\quad  \norm{Ku_{\lambda_2}}_{\ell^1(\R^m)} = 1,
\end{alignat*}
and
\begin{align*}
J(u_{\lambda_1})K^*Ku_{\lambda_1} &= (0,1/2,1,1/2,0)^T \in\text{sign}(u_{\lambda_1}), \\ J(u_{\lambda_2})K^*Ku_{\lambda_2} &= (1/2,1,0,-1,-1/2)^T \in\text{sign}(u_{\lambda_2}).
\end{align*}
The two singular vectors are $K$-orthogonal but they do not satisfy the (\ref{eq: subdiff_cond}) condition since
\begin{align*}
J(u_{\lambda_1})K^*Ku_{\lambda_1} + J(u_{\lambda_2})K^*Ku_{\lambda_2} = (1/2,\textcolor{red}{3/2},1,-1/2,-1/2)^T \notin \partial J(0)=[-1,1].
\end{align*}
This example also shows that Proposition \ref{prop: J_linearity} only holds one way since for $c_1,c_2\geq 0$ we have
\begin{align*}
J(c_1u_{\lambda_1}+c_2u_{\lambda_2}) &= \norm{(0,c_2/\sqrt{2},c_1,-c_2/\sqrt{2},0)^T}_{\ell^1(\R^m)} = c_1 + 2c_2/\sqrt{2} \\
&= c_1\cdot 1 + c_2\cdot 2/\sqrt{2} = c_1J(u_{\lambda_1}) + c_2J(u_{\lambda_2}).
\end{align*}
Hence the linearity of $J$ in the singular vectors does not imply that the (\ref{eq: subdiff_cond}) condition is satisfied. In Figure \ref{fig: l1conv_example} we see the two singular vectors and their corresponding subgradients. We see that the subgradients contain values between $-1$ and $1$ and are exactly $-1$ at the negative peaks and $1$ at the positive peaks. In Figure \ref{subfig: l1conv_example_3} we see the sum of the singular vectors that contains three peaks. The sum of the subgradients contains values above one and is therefore not in the subdifferential of $J$ at zero. The problem is that the two positive peaks of the singular vectors get too close, which makes the sum of the subgradients exceed one. \\

\begin{figure}[!h]
\begin{center}
\subfloat[First singular vector and subgradient.]{\includegraphics[width=0.49\textwidth]{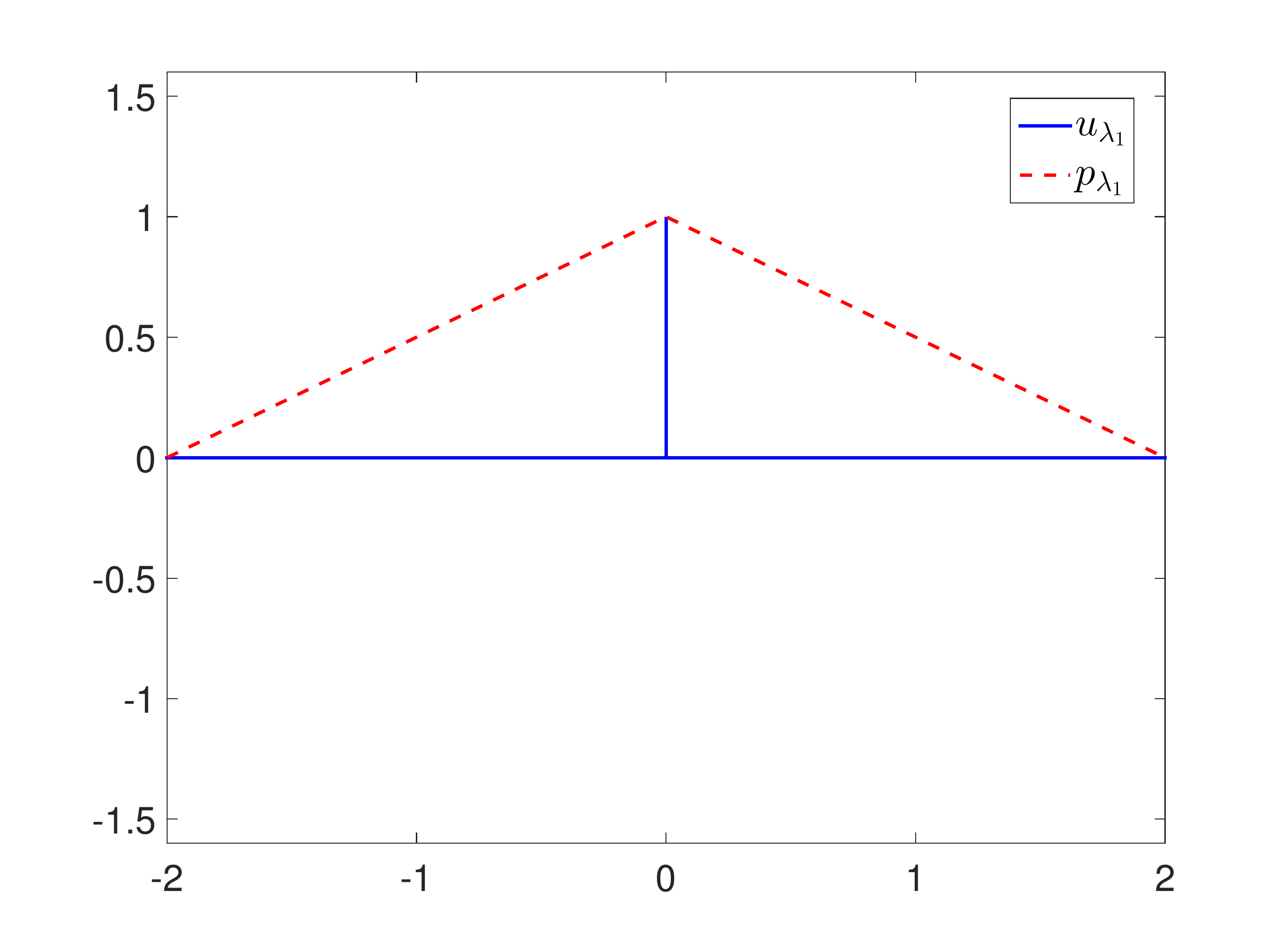}\label{subfig: l1conv_example_1}}\;
\subfloat[Second singular vector and subgradient.]{\includegraphics[width=0.49\textwidth]{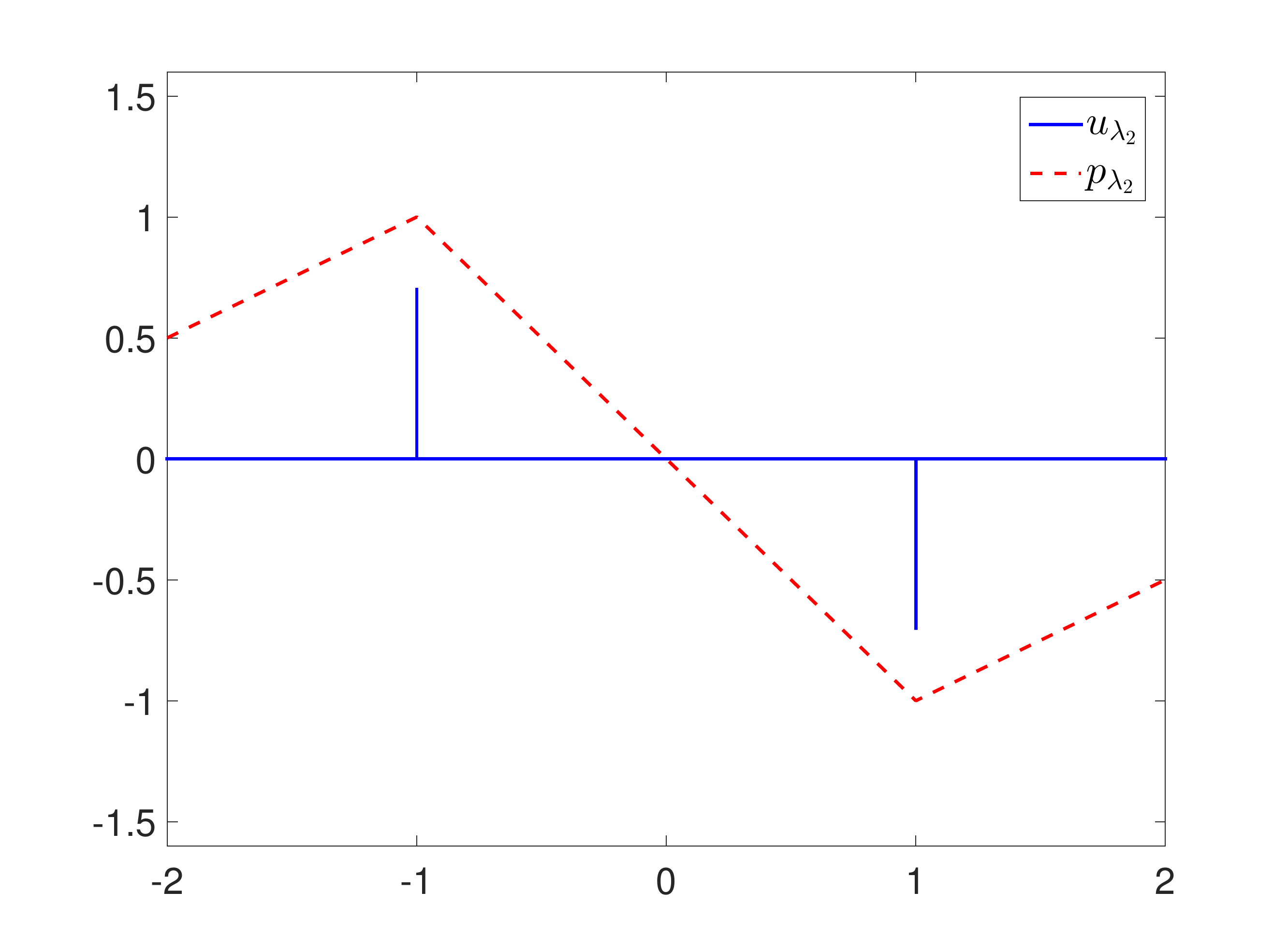}\label{subfig: l1conv_example_2}}\\
\subfloat[Sum of singular vectors and subradients.]{\includegraphics[width=0.49\textwidth]{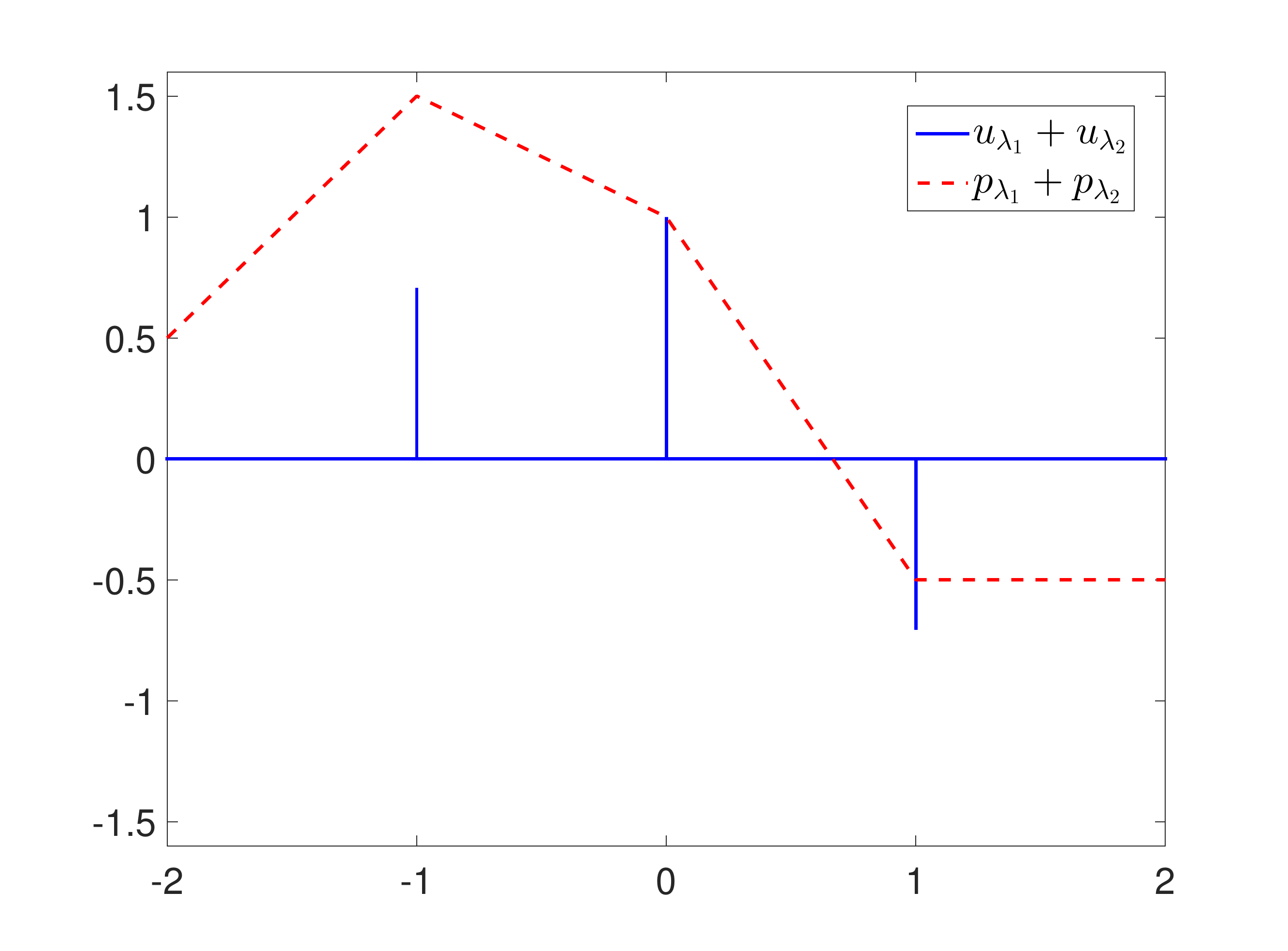}\label{subfig: l1conv_example_3}}\\
\end{center}
\vspace{0.5cm}
\caption{Convolution example. In Figures \ref{subfig: l1conv_example_1} and \ref{subfig: l1conv_example_2} we see two $K$-orthogonal singular vectors. Their subgradients reach $1$, respectively $-1$, for positive, respectively negative, peaks. Between the peaks the value is in $[-1,1]$. In Figure \ref{subfig: l1conv_example_3} we see the sum of the two singular vectors and the sum of the subgradients. The (\ref{eq: subdiff_cond}) condition is not satisfied since the sum of the subgradients exceeds one. The problem is the two positive peaks of the singular vectors getting too close.}
\label{fig: l1conv_example}
\end{figure}

\end{example}
\begin{example}
We consider $K = I:BV([0,1])\to L^2([0,1])$, $J(u) = \tv_*(u)$,  and the two Haar wavelets $u_{\lambda_1}$ and $u_{\lambda_2}$ as defined in Example \ref{exm:tvsingvec}. By adding these two singular vectors we get
\begin{align*}
(u_{\lambda_1}+u_{\lambda_2})(x) = \left\{
        \begin{array}{ll}
            1+\sqrt{2}, & 0\leq x < 1/4,\\
            1-\sqrt{2}, & 1/4 \leq x < 1/2,\\
            -1, & 1/2 \leq x \leq 1, \\
            0, & \text{otherwise}.
        \end{array}
    \right.
\end{align*}
The total variation of the sum is then
\begin{align*}
\tv_*(u_{\lambda_1}+u_{\lambda_2}) &= \underset{\substack{\varphi\in C^{\infty}([0,1]),\\ \norm{\varphi}_{L^{\infty}([0,1])}\leq 1}}{\sup} \int_0^1 (u_{\lambda_1}+u_{\lambda_2})(x)\varphi'(x)\, dx, \\
&= \underset{\substack{\varphi\in C^{\infty}([0,1]),\\ \norm{\varphi}_{L^{\infty}([0,1])}\leq 1}}{\sup} \left\{(1+\sqrt{2})\int_{0}^{1/4} \varphi'(x)\, dx + (1-\sqrt{2})\int_{1/4}^{1/2} \varphi'(x)\, dx \right. \\
&\left. \qquad\qquad\qquad - \int_{1/2}^{1} \varphi'(x)\, dx \right\} \\
&= \underset{\substack{\varphi\in C^{\infty}([0,1]),\\ \norm{\varphi}_{L^{\infty}([0,1])}\leq 1}}{\sup} \left\{ -(1+\sqrt{2})\varphi(0) + 2\sqrt{2}\varphi(1/4) + (2-\sqrt{2})\varphi(1/2) \right. \\
&\left. \qquad\qquad\qquad - \varphi(1) \vphantom{\sqrt{2}} \right\} \\
&\leq 1+\sqrt{2}+2\sqrt{2} +2-\sqrt{2} +1 = 4 + 2\sqrt{2} \\
&= \tv_*(u_{\lambda_1}) + \frac{1}{2} \tv_*(u_{\lambda_2}) \\
&< \tv_*(u_{\lambda_1}) + \tv_*(u_{\lambda_2}).
\end{align*}
By Proposition \ref{prop: J_linearity}, $u_{\lambda_1}$ and $u_{\lambda_2}$ do not satisfy the (\ref{eq: subdiff_cond}) condition. In Figure \ref{subfig: TV_example_1} we see the two Haar wavelets plotted. In Figure \ref{subfig: TV_example_2} we see the sum of the two Haar wavelets. At $x = 1/2$ we have a second jump downwards. This is at the place where the two Haar wavelets have overlapping jumps in opposite directions. That is exactly what makes the $\tv_*$-functional non-linear in the singular vectors. \\

\begin{figure}
\begin{center}
\subfloat[Two Haar wavelets.]{\includegraphics[width=0.49\textwidth]{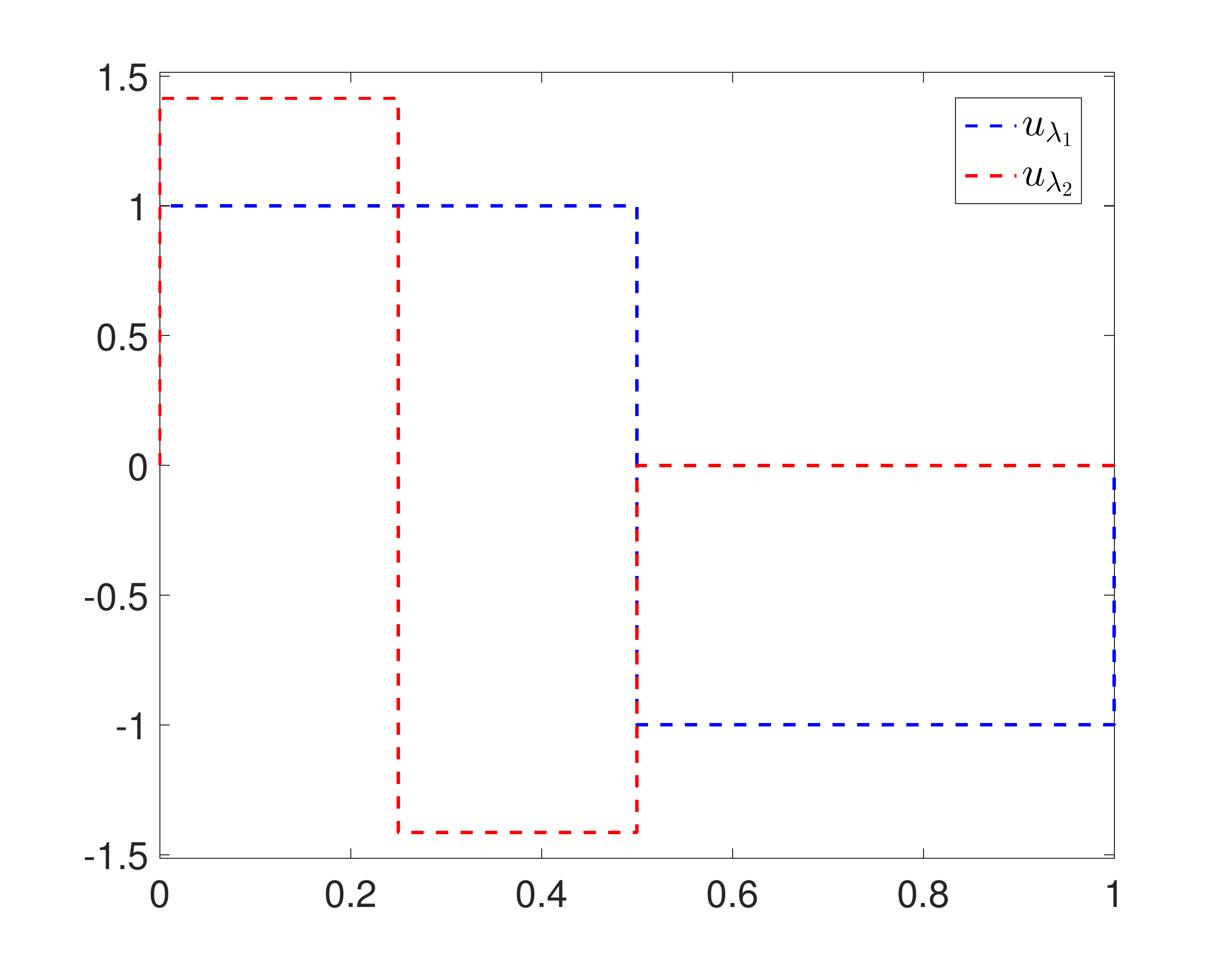}\label{subfig: TV_example_1}}
\subfloat[Sum of the two Haar wavelets.]{\includegraphics[width=0.49\textwidth]{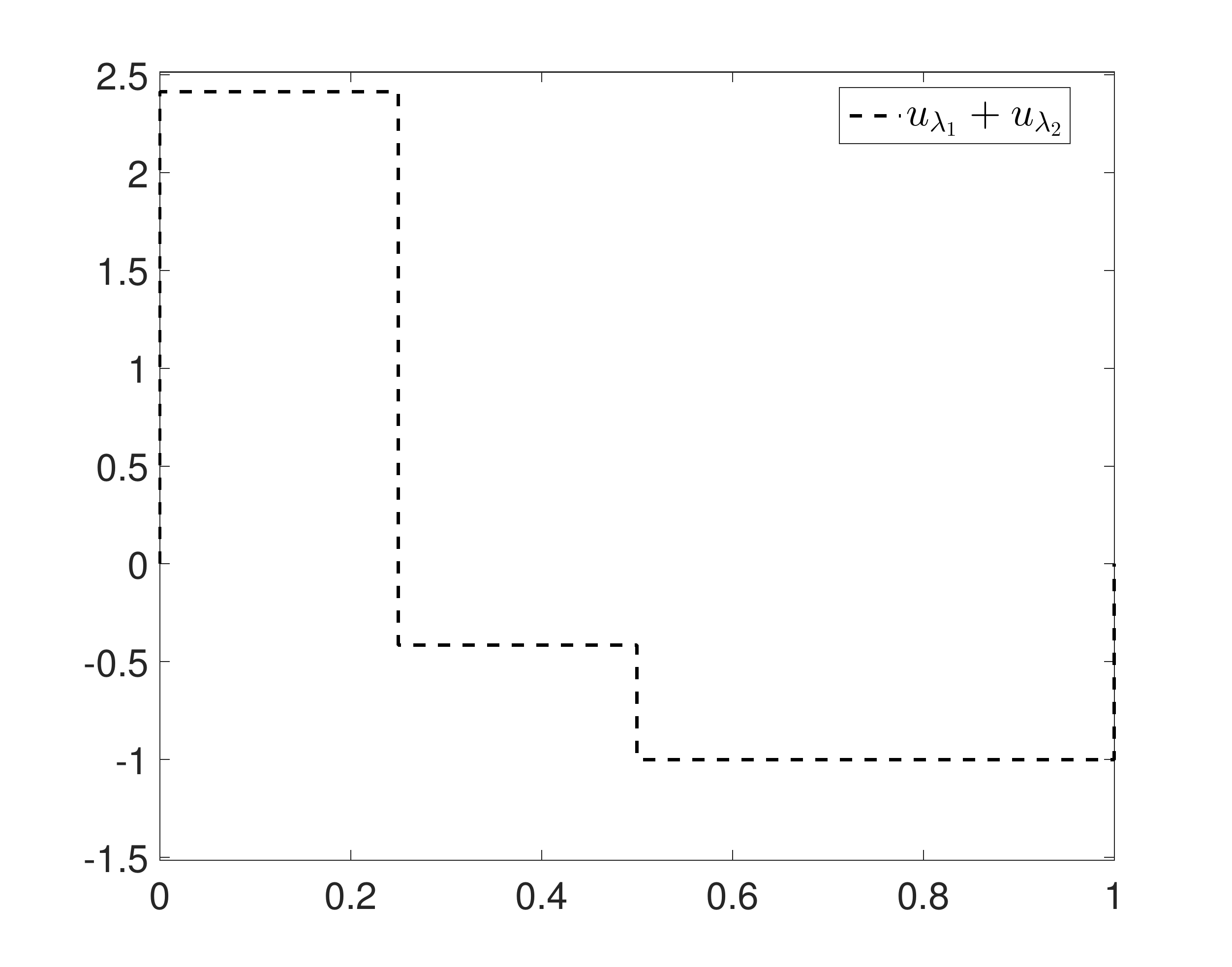}\label{subfig: TV_example_2}}
\end{center}
\vspace{0.5cm}
\caption{$\tv_*$-denoising example. In Figure \ref{subfig: TV_example_1} we see two Haar wavelets which are orthogonal singular vectors for the $\tv_*$-functional and the embedding operator $I: BV([0,1])\to L^2([0,1])$. The singular vectors do not satisfy the (\ref{eq: subdiff_cond}) condition by Proposition \ref{prop: J_linearity}, since $\tv_*(u_{\lambda_1}+u_{\lambda_2}) < \tv_*(u_{\lambda_1}) + \tv_*(u_{\lambda_2})$.}
\label{fig: TV_example}
\end{figure}
\end{example}

In the following we are going to provide an alternative \textit{dual norm formulation} of the (\ref{eq: subdiff_cond}) condition under a norm-inequality assumption on $J$.
\begin{definition}[Norm-inequality Assumption] 
\label{definition: NormIneqAssumption}
$J$ satisfies the \textit{norm-inequality assumption} if there exists a positive constant $c_0>0$ such that
\begin{align}
c_0\norm{u}_{\uc} \leq J(u),\quad\forall u\in\ker(J)^{\perp}.
\label{eq: NormIneqAssumption}
\end{align}
\end{definition}
Under the norm-inequality assumption we are able to show a relation between the subdifferential at zero and the dual norm:
\begin{proposition} 
\label{prop: DualNormFormOfSubdiff}
Let $J$ be absolutely one-homogeneous and satisfy the norm-inequality assumption (\ref{eq: NormIneqAssumption}). Then
\begin{align*}
\norm{\omega^*}_{\uc^*} \leq c_0 \text{ and } \ip{\omega^*}{v} = 0,\; \forall v\in\ker(J) \quad \Rightarrow \quad \omega^*\in\partial J(0).
\end{align*}
\end{proposition}
\begin{proof}
Let $\norm{\omega^*}_{\uc^*} \leq c_0$ and $\ip{\omega^*}{v}=0$ for all $v\in\ker(J)$. We need to show $\ip{\omega^*}{\omega} \leq J(\omega)$ for all $\omega\in\uc$. We can already exclude all $\omega\in\ker(J)$ since in this case we have $\ip{\omega^*}{\omega} = 0$ by assumption and $J(\omega) = 0$. For $\omega\in\ker(J)^{\perp}$ we get
\begin{align*}
\ip{\omega^*}{\omega} \leq \norm{\omega^*}_{\uc^*}\norm{\omega}_{\uc} \leq c_0\norm{\omega}_{\uc} \leq J(\omega)
\end{align*}
by (\ref{eq: NormIneqAssumption}). Hence for $\omega = \omega_1 + \omega_2$ with $\omega_1\in\ker(J)$ and $\omega_2\in\ker(J)^{\perp}$ we have
\begin{align*}
\ip{\omega^*}{\omega} = \ip{\omega^*}{\omega_2} \leq J(\omega_2) = J(\omega),
\end{align*}
where the last equality follows from (\ref{eq: J_kernelinvariance}).
\end{proof}
Now we see that if
\begin{align}
\norm{\sum_{j=1}^{k} \lambda_jK^*Ku_{\lambda_j}}_{\uc^*} \leq c_0
\label{eq: DualNormFormOfSUB0}
\end{align}
for $k\in\{1;n\}$ then the (\ref{eq: subdiff_cond}) condition is satisfied, since 
\begin{align*}
\ip{\sum_{j=1}^{k} \lambda_jK^*Ku_{\lambda_j}}{v} = 0, \quad \forall v\in\ker(J)
\end{align*}
is already satisfied by (\ref{eq: SubgradientPerpToKernelJ}) and $\lambda_jK^*Ku_{\lambda_j}\in\partial J(u_{\lambda_j})$ for $j=1,...,k$. Hence we can use (\ref{eq: DualNormFormOfSUB0}) as an alternative to the (\ref{eq: subdiff_cond}) condition for numerical implementations when the norm-inequality assumption is satisfied. \\

\begin{remark}
\label{remark: SubdiffEquivDualNorm}
Note that if $J(u) = \norm{u}_{\uc}$ then
\begin{align*}
\norm{\omega^*}_{\uc^*} \leq 1 \text{ and } \ip{\omega^*}{v} = 0,\;\forall v\in\ker(J) \quad\Rightarrow\quad \omega^*\in\partial J(0)
\end{align*}
by Proposition \ref{prop: DualNormFormOfSubdiff} with $c_0 = 1$ and for $\omega^*\in\partial J(0)$ we have
\begin{align*}
\norm{\omega^*}_{\uc^*} = \sup_{\norm{v}_{\uc}\leq 1} \ip{\omega^*}{v} \leq \sup_{\norm{v}_{\uc}\leq 1} J(v) = \sup_{J(v)\leq 1} J(v) = 1.
\end{align*}
Hence in this case the (\ref{eq: subdiff_cond}) condition is equivalent to
\begin{align*}
\norm{\sum_{j=1}^{n} \lambda_jK^*Ku_{\lambda_j}}_{\uc^*} \leq 1.
\end{align*}
\end{remark}

\begin{example}
\label{exampel: SUBneverSatisfied}
With this final \eqref{eq: subdiff_cond}-example we want to highlight that given the orthogonality condition \eqref{eq: orthogonality_cond}, the one-homogeneous analogue to Showalter's method, i.e. the ISS flow \eqref{eq: ISSflow} with $J(u) = \| u \|_{L^2(\Omega)}$, $\uc = L^2(\Omega)$, satisfies the (\ref{eq: subdiff_cond}) condition only for $k = 1$. We see this by considering \eqref{eq: subdiff_cond} in this particular case. We have
\begin{align*}
\left\| \sum_{j = 1}^k \lambda_j K^* K u_{\lambda_j} \right\|_{L^2(\Omega)} = \left\| \sum_{j = 1}^k \frac{u_{\lambda_j}}{\| u_{\lambda_j} \|_{L^2(\Omega)}} \right\|_{L^2(\Omega)} \, \text{,}
\end{align*}
due to $\partial \| u \|_{L^2(\Omega)} = u/ \| u \|_{L^2(\Omega)}$ for $\| u \|_{L^2(\Omega)} \neq 0$. Due to the $K$-orthogonality in combination with the singular vector condition \eqref{eq:gensingvec} we observe $0 = \lambda_i \langle Ku_{\lambda_i}, Ku_{\lambda_j} \rangle = \langle u_{\lambda_i}, u_{\lambda_j} \rangle / \| u_{\lambda_i} \|_{L^2(\Omega)}$, which implies $\langle u_{\lambda_i}, u_{\lambda_j} \rangle = 0$. Hence, we have
\begin{align*}
\left\| \sum_{j = 1}^k \frac{u_{\lambda_j}}{\| u_{\lambda_j} \|_{L^2(\Omega)}} \right\|_{L^2(\Omega)}^2 &= \left\langle \sum_{i = 1}^k \frac{u_{\lambda_i}}{\| u_{\lambda_i} \|_{L^2(\Omega)}}, \sum_{j = 1}^k \frac{u_{\lambda_j}}{\| u_{\lambda_j} \|_{L^2(\Omega)}} \right\rangle\\
&= \sum_{i = 1}^k \frac{\langle u_{\lambda_i}, u_{\lambda_i}\rangle}{\| u_{\lambda_i} \|_{L^2(\Omega)} \| u_{\lambda_i} \|_{L^2(\Omega)}} = k \, \text{,}
\end{align*} 
and therefore 
\begin{align*}
\left\| \sum_{j = 1}^k \lambda_j K^* K u_{\lambda_j} \right\|_{L^2(\Omega)} = \sqrt{k} \geq 1 \, ,
\end{align*}
for $k \in \N$. Thus, according to Remark \ref{remark: SubdiffEquivDualNorm}, the \eqref{eq: subdiff_cond} condition is only satisfied for $k = 1$.
\end{example}

\subsection{Inverse Scale Space Decomposition}
We now provide our main result which states that the ISS flow gives a perfect decomposition into the generalised singular vectors representing the data of the flow given that (\ref{eq: orthogonality_cond}) and (\ref{eq: subdiff_cond}) are satisfied.

\begin{theorem}[Inverse Scale Space Decomposition] 
\label{thm: ExactReconstruction}
Let $\{u_{\lambda_j}\}_{j=1}^{n}$ be a set of $K$-normalised singular vectors of $J$ with corresponding singular values $\{\lambda_j\}_{j=1}^{n}$. Assume that they satisfy the orthogonality condition (\ref{eq: orthogonality_cond}) and the (\ref{eq: subdiff_cond}) condition. Then, if the data $f$ is given by $f = \sum_{j=1}^{n} \gamma_j Ku_{\lambda_j}$ for positive constants $\{\gamma_j\}_{j=1}^{n}$, a solution of the inverse scale space flow (\ref{eq: ISSflow}) is given by
\begin{align*}
u(t) = \left\{
        \begin{array}{ll}
            0, & \quad 0\leq t < t_1, \\[0.1cm]
            \sum_{j=1}^{k} \gamma_j u_{\lambda_j}, & \quad t_k \leq t < t_{k+1} \,\text{ for }\, k=1,...,n-1, \\[0.2cm]
            \sum_{j=1}^{n} \gamma_j u_{\lambda_j}, & \quad t_n \leq t,
        \end{array}
    \right.
\end{align*}
where $t_k = \lambda_k/\gamma_k$ and $t_k < t_{k+1}$.
\end{theorem}

\begin{proof}
The proof is divided into the following three steps corresponding to the three different cases for $u(t)$: \\

\setlength{\leftskip}{1cm}

\textcolor{blue}{\textbf{Step 1:}} $0 \leq t < t_1$.\\
\textcolor{orange}{\textbf{Step 2:}} $t_k \leq t < t_{k+1}$ for $k=1,...,n-1$ \\
\textcolor{red}{\textbf{Step 3:}} $t_n \leq t$. \\

\setlength{\leftskip}{0pt}

For every step we have to show that the solution satisfies the inverse scale space flow. \\

\textcolor{blue}{\textbf{Step 1:}} Let $0 \leq t < t_1$. Then $u(t) = 0$, and from the inverse scale space flow we get
\begin{align*}
\partial_t p(t) = K^* f, \quad p(0) = 0.
\end{align*}
Integrating $\int_{0}^{t} \partial_\tau p(\tau) d\tau$ yields
\begin{align*}
p(t) = tK^*f = t\left(\sum_{j=1}^{n} \frac{\gamma_j}{\lambda_j}p_{\lambda_j}\right),
\end{align*}
where we have used the notation $p_{\lambda_j} = \lambda_jK^*Ku_{\lambda_j}$. We need to show that $p(t)\in\partial J(u(t)) = \partial J(0)$. Using the characterisation (\ref{eq: subdiff_characterization}) of the subdifferential we need to show that $\ip{p(t)}{v}\leq J(v)$ for all $v\in\uc$. We have
\begin{align*}
\ip{p(t)}{v} = t\ip{\sum_{j=1}^{n} \frac{\gamma_j}{\lambda_j}p_{\lambda_j}}{v}.
\end{align*}
Hence we need to show
\begin{align}
t\ip{\sum_{j=1}^{n} \frac{\gamma_j}{\lambda_j}p_{\lambda_j}}{v} \leq J(v), \; \forall v\in \uc.
\label{eq: t1_ineq}
\end{align}
If $\ip{\sum_{j=1}^{n} \frac{\gamma_j}{\lambda_j}p_{\lambda_j}}{v}\leq 0$ then inequality (\ref{eq: t1_ineq}) is trivially satisfied since $J(v) \geq 0$ for all $v\in\uc$. If $\ip{\sum_{j=1}^{n} \frac{\gamma_j}{\lambda_j}p_{\lambda_j}}{v} > 0$ then
\begin{align*}
t\ip{\sum_{j=1}^{n} \frac{\gamma_j}{\lambda_j}p_{\lambda_j}}{v} < t_1\ip{\sum_{j=1}^{n} \frac{\gamma_j}{\lambda_j}p_{\lambda_j}}{v} = \ip{p_{\lambda_1}}{v} + t_1\ip{\sum_{j=2}^{n} \frac{\gamma_j}{\lambda_j}p_{\lambda_j}}{v}
\end{align*}
If $\ip{\sum_{j=2}^{n} \frac{\gamma_j}{\lambda_j}p_{\lambda_j}}{v}\leq 0$ inequality (\ref{eq: t1_ineq}) is satisfied since $p_{\lambda_1} \in \partial J(0)$. If $\ip{\sum_{j=2}^{n} \frac{\gamma_j}{\lambda_j}p_{\lambda_j}}{v} > 0$ then
\begin{align*}
\ip{p_{\lambda_1}}{v} &+ t_1\ip{\sum_{j=2}^{n} \frac{\gamma_j}{\lambda_j}p_{\lambda_j}}{v} < \ip{p_{\lambda_1}}{v} + t_2\ip{\sum_{j=2}^{n} \frac{\gamma_j}{\lambda_j}p_{\lambda_j}}{v} \\
&= \ip{p_{\lambda_1} + p_{\lambda_2}}{v} + t_2\ip{\sum_{j=3}^{n} \frac{\gamma_j}{\lambda_j}p_{\lambda_j}}{v}
\end{align*}
If $\ip{\sum_{j=3}^{n} \frac{\gamma_j}{\lambda_j}p_{\lambda_j}}{v}\leq 0$ then inequality (\ref{eq: t1_ineq}) is satisfied since $p_{\lambda_1}+p_{\lambda_2}\in\partial J(0)$ by the (\ref{eq: subdiff_cond}) condition. If $\ip{\sum_{j=3}^{n} \frac{\gamma_j}{\lambda_j}p_{\lambda_j}}{v} > 0$ we continue the process. The process terminates when we reach
\begin{align*}
\ip{\sum_{j=1}^{n} p_{\lambda_j}}{v},
\end{align*}
which is less than or equal to $J(v)$ since $\sum_{j=1}^{n} p_{\lambda_j} \in \partial J(0)$ by the (\ref{eq: subdiff_cond}) condition. In any case we have shown that $p(t)\in\partial J(0)$ for $0\leq t<t_1$. \\

\textcolor{orange}{\textbf{Step 2:}} Let $t_k \leq t < t_{k+1}$. Then $u(t) = \sum_{j=1}^{k} \gamma_j u_{\lambda_j}$ and from the inverse scale space flow we get
\begin{align*}
\partial_t p(t) = K^*\left(f-\sum_{j=1}^{k} \gamma_j Ku_{\lambda_j}\right) = \sum_{j=k+1}^{n} \gamma_jK^*K u_{\lambda_j} = \sum_{j=k+1}^{n} \frac{\gamma_j}{\lambda_j}p_{\lambda_j}.
\end{align*}
A continuous extension of $p$ at $t = t_k$ is given by
\begin{align*}
p(t) = \sum_{j=1}^{k} p_{\lambda_j} + t\left(\sum_{j=k+1}^{n} \frac{\gamma_j}{\lambda_j}p_{\lambda_j}\right).
\end{align*}
We require $p(t)\in\partial J(u(t)) = \partial J \left(\sum_{j=1}^{k} \gamma_j u_{\lambda_j}\right)$. Using the characterisation (\ref{eq: subdiff_characterization}) of the subdifferential we need to show
\begin{align}
\ip{p(t)}{\sum_{j=1}^{k} \gamma_j u_{\lambda_j}} = J\left(\sum_{j=1}^{k} \gamma_j u_{\lambda_j}\right) \quad\text{and}\quad \ip{p(t)}{v} \leq J(v),\;\forall v\in\uc.
\label{eq:criteria}
\end{align}
For the first criterion of (\ref{eq:criteria}) we get
\begin{align*}
\ip{p(t)}{\sum_{j=1}^{k} \gamma_j u_{\lambda_j}} &= \ip{\sum_{j=1}^{k} p_{\lambda_j}}{\sum_{j=1}^{k} \gamma_j u_{\lambda_j}} + t\ip{\sum_{j=k+1}^{n}\frac{\gamma_j}{\lambda_j}p_{\lambda_j}}{\sum_{j=1}^{k} \gamma_j u_{\lambda_j}} \\
&= \sum_{j=1}^{k} \gamma_j J(u_{\lambda_j}) = J\left(\sum_{j=1}^{k} \gamma_j u_{\lambda_j}\right).
\end{align*}
We have used the formula for $p(t)$, the properties of singular vectors, the orthogonality condition (\ref{eq: orthogonality_cond}) and for the last equality Proposition \ref{prop: J_linearity}. For the second criterion of (\ref{eq:criteria}) we can use the same process as for \textcolor{blue}{\textbf{Step 1}}. \\

\textcolor{red}{\textbf{Step 3:}} Let $t_n \leq t$. Then $u(t) = \sum_{j=1}^{n} \gamma_j u_{\lambda_j}$ and from the inverse scale space flow we get
\begin{align*}
\partial_t p(t) = K^*(f - \sum_{j=1}^{n} \gamma_j Ku_{\lambda_j}) = 0.
\end{align*}
A continuous extension of $p$ at $t = t_n$ is
\begin{align*}
p(t) = \sum_{j=1}^{n} p_{\lambda_j}.
\end{align*}
We need to show that $p(t)\in\partial J(u(t)) = \partial J\left(\sum_{j=1}^{n} \gamma_j u_{\lambda_j}\right)$. This follows directly from Proposition \ref{prop: subdifferential_linearity_equivalences}.
\end{proof}
\begin{remark}
\label{remark: n=2}
Note that for $n=2$ the (\ref{eq: subdiff_cond}) condition is both necessary and sufficient under the orthogonality condition (\ref{eq: orthogonality_cond}). In this case we can compute
\begin{align*}
p(t) = \left\{
        \begin{array}{ll}
            tK^*f, & \quad 0\leq t < t_1, \\[0.1cm]
            p_{\lambda_1} + t\gamma_2K^*Ku_{\lambda_2}, & \quad t_1 \leq t < t_2, \\[0.1cm]
            p_{\lambda_1}+p_{\lambda_2}, & \quad t_2 \leq t.
        \end{array}
    \right.
\end{align*}
Hence we need $p(t)\in\partial J(u(t))$ for $t_2 \leq t$, i.e.
\begin{align*}
p_{\lambda_1}+p_{\lambda_2} \in \partial J(\gamma_1 u_{\lambda_1} + \gamma_2u_{\lambda_2}) \subset \partial J(0).
\end{align*}
This shows that the (\ref{eq: subdiff_cond}) condition is necessary. In fact, for arbitrary $n$ we always need $\sum_{j=1}^{n} p_{\lambda_j} \in \partial J(0)$ for the decomposition to happen. This is, however, not enough, as we will see in Example \ref{example: PartialSUB0Necessary}. \\
\end{remark}
\begin{remark}
If $u_{\lambda_j}$ is a singular vector then also $-u_{\lambda_j}$ is a singular vector. Hence if $\gamma_j < 0$, we can just pull the negative sign of $\gamma_j$ onto $u_{\lambda_j}$ and regard this as a singular vector with a positive contribution to $f$. Then Theorem \ref{thm: ExactReconstruction} still holds true with $t_k = \lambda_k/|\gamma_k|$ and the (\ref{eq: subdiff_cond}) condition being replaced by
\begin{align*}
\sum_{j=1}^{k} \frac{\gamma_j}{|\gamma_j|}\lambda_jK^*Ku_{\lambda_j}\in\partial J(0),\; \forall k\in\{1;n\}.
\end{align*}
\end{remark}
\begin{remark}
In Example \ref{exampel: SUBneverSatisfied} we have shown that the (\ref{eq: subdiff_cond}) condition is never satisfied for $n > 1$ for $J(u) = \norm{u}_{L^2(\Omega)}$, $\uc = L^2(\Omega)$. Since by Remark \ref{remark: n=2} the full sum over the subgradients of the singular vectors always has to be in the subdifferential of $J$ at zero in order for the ISS flow to give a decomposition into the $K$-orthogonal singular vectors representing $f$, we then conclude that the inverse scale space decomposition can never happen in this case without re-weighting its coefficients.\\
\end{remark}

We give a comparison to a decomposition result in \cite[Theorem 2]{gilboa2015semi} under the earlier mentioned (LIS) condition \cite[Definition 7]{gilboa2015semi} and orthogonality condition (FO) \cite[Definition 5]{gilboa2015semi}. As previously mentioned the (LIS) condition requires linearity in the subdifferential for both positive and negative coefficients. The orthogonality condition is formulated a bit differently but boils down to the same as (\ref{eq: orthogonality_cond}) and is actually automatically satisfied under the (LIS) condition when $J$ is absolutely one-homogeneous. It is then shown that the (forward) scale space flow
\begin{align*}
\partial_t u(t) = -p(t), \quad p(t)\in\partial J(u(t)), \quad u(0) = f
\end{align*}
gives a decomposition into singular vectors under the (LIS) and (FO) conditions when $f$ is given as the sum of two singular vectors and $J$ is one-homogeneous. The results are stated for the forward scale space flow and for denoising problems where $K$ is the identity and involve decomposition into two singular vectors only. \\

To emphasise that the partial sum conditions of the (\ref{eq: subdiff_cond}) condition are necessary we look at the following example:
\begin{example}
\label{example: PartialSUB0Necessary}
Let $J(u) = \norm{u}_{\ell^1(\R^m)}$ with $\partial J(u) = \text{sign}(u)$, see example \ref{exm:dictionaryexm}. We reuse the convolution operator $K:\ell^1(\R^m) \to \ell^2(\R^m)$ from Example \ref{example: l1_convolution}. Define the vectors
\begin{align*}
&u_{\lambda_1} = (0,0,0,1,-1,0,0,0,0)^T, &&u_{\lambda_2} = (0,0,0,0,-1,1,0,0,0)^T, \\
&u_{\lambda_3} = \frac{1}{\sqrt{2}}(0,0,0,1,0,1,0,0,0)^T, &&u_{\lambda_4} = (0,-1,0,0,0,0,0,0,0)^T, \\
&u_{\lambda_5} = (0,0,0,0,0,0,0,-1,0)^T.
\end{align*}
Using the notation $p_{\lambda_j} = \lambda_jK^*Ku_{\lambda_j}$ we get
\begin{align*}
&p_{\lambda_1} = (0,0,1,1,-1,-1,0,0,0)^T\in\partial J(u_{\lambda_1}), \\ &p_{\lambda_2} = (0,0,0,-1,-1,1,1,0,0)^T\in\partial J(u_{\lambda_2}), \\
&p_{\lambda_3} = (0,0,1/2,1,1,1,1/2,0,0)^T\in\partial J(u_{\lambda_3}), \\ &p_{\lambda_4} = (-1/2,-1,-1/2,0,0,0,0,0,0)^T \in\partial J(u_{\lambda_4}), \\
&p_{\lambda_5} = (0,0,0,0,0,0,-1/2,-1,-1/2)^T \in\partial J(u_{\lambda_5}).
\end{align*}
Hence $\{u_{\lambda_j}\}_{j=1}^5$ is a set of five singular vectors. In Figure \ref{fig: l1conv2_example_singvecs} we see all the singular vectors and their corresponding subgradients. We observe that two peaks of the same sign cannot get too close if we want to obtain a singular vector since the convolution kernel spreads out the peaks. We also see that the subgradients have their maximum magnitude $1$ at the peaks. \\

\begin{figure}
\begin{center}
\subfloat[\textcolor{blue}{$u_{\lambda_1}$} and \textcolor{red}{$p_{\lambda_1}$}.]{\includegraphics[width=0.32\textwidth]{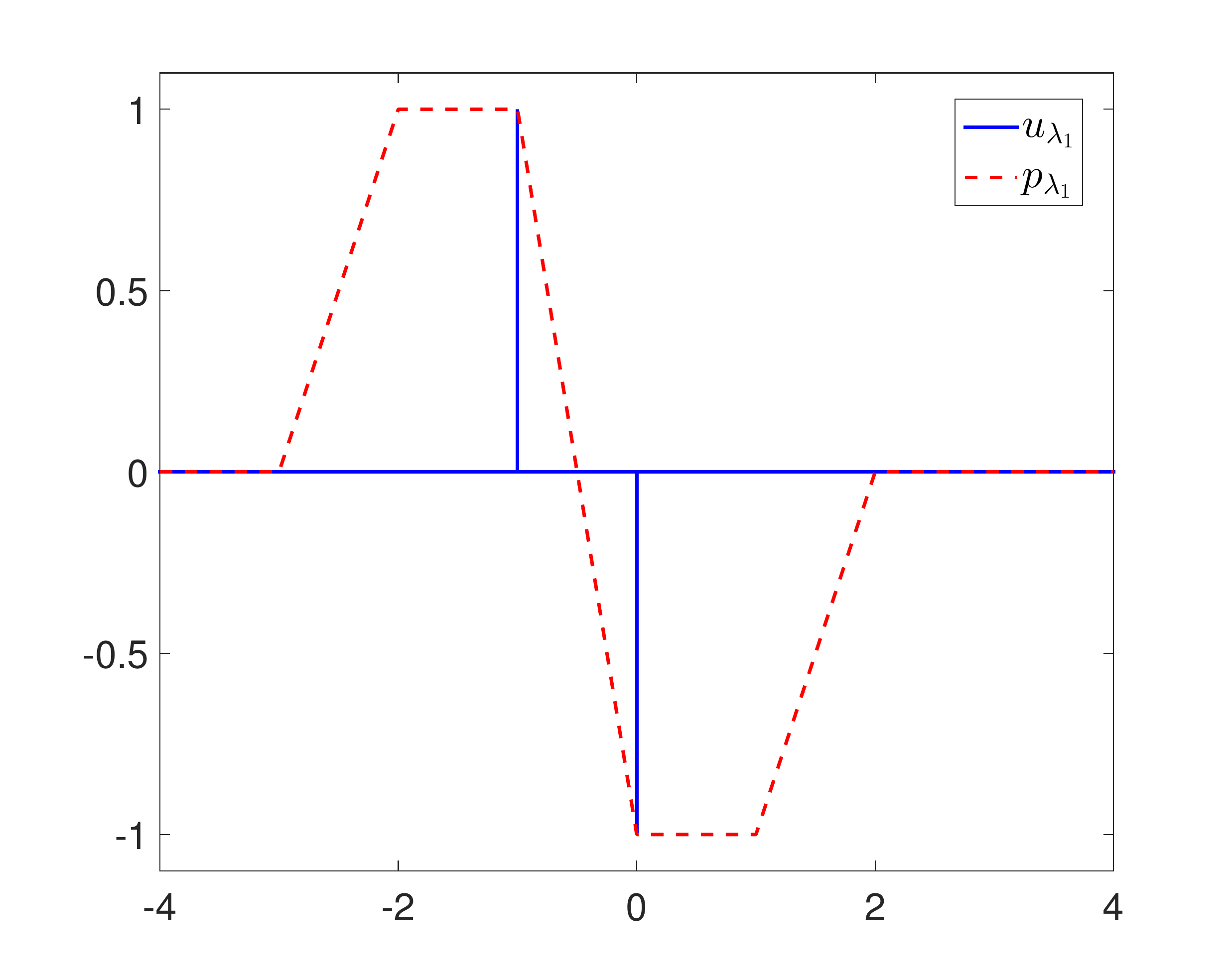}\label{subfig: l1conv2_example_singvecs_1}}
\subfloat[\textcolor{blue}{$u_{\lambda_2}$} and \textcolor{red}{$p_{\lambda_2}$}.]{\includegraphics[width=0.32\textwidth]{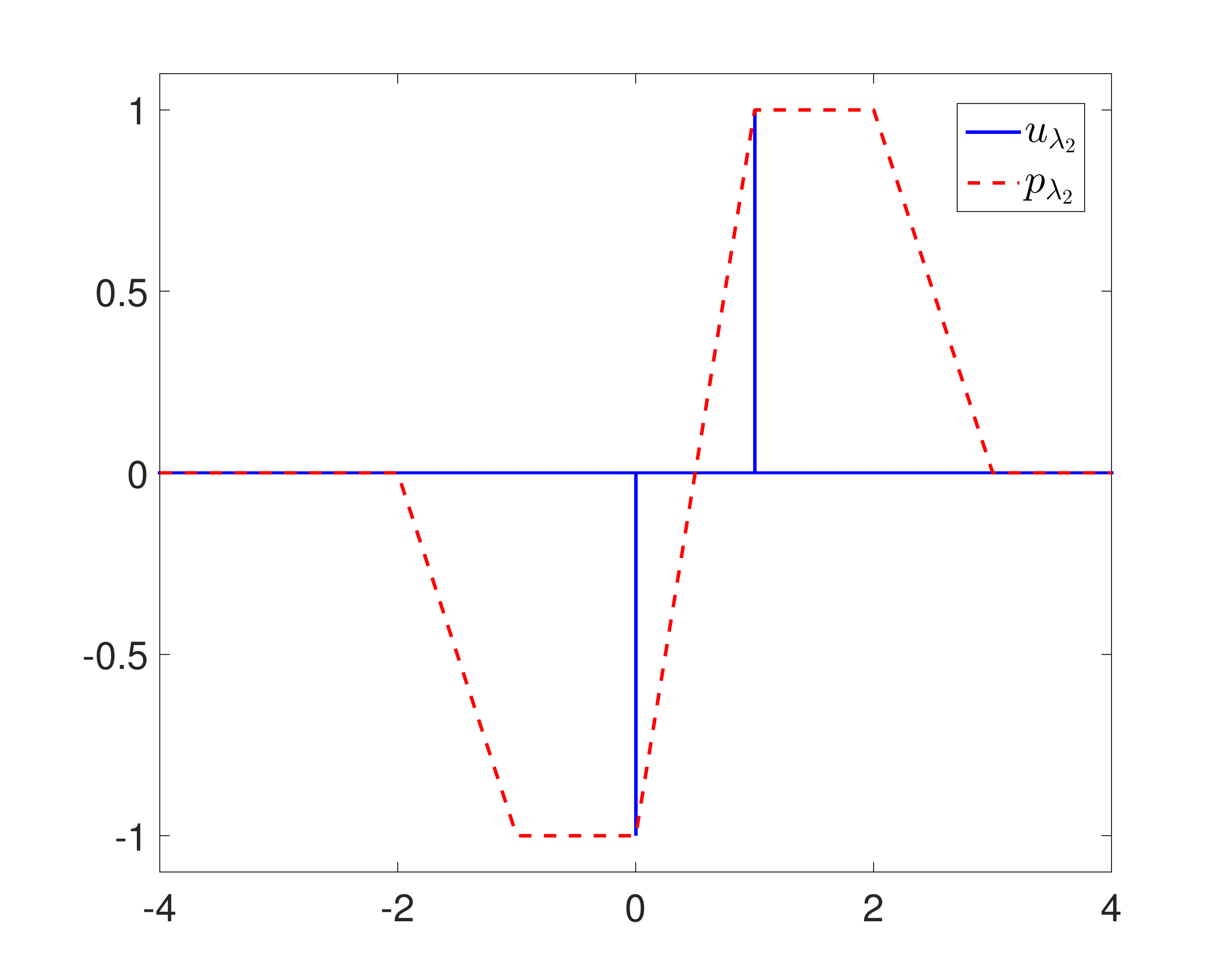}\label{subfig: l1conv2_example_singvecs_2}}
\subfloat[\textcolor{blue}{$u_{\lambda_3}$} and \textcolor{red}{$p_{\lambda_3}$}.]{\includegraphics[width=0.32\textwidth]{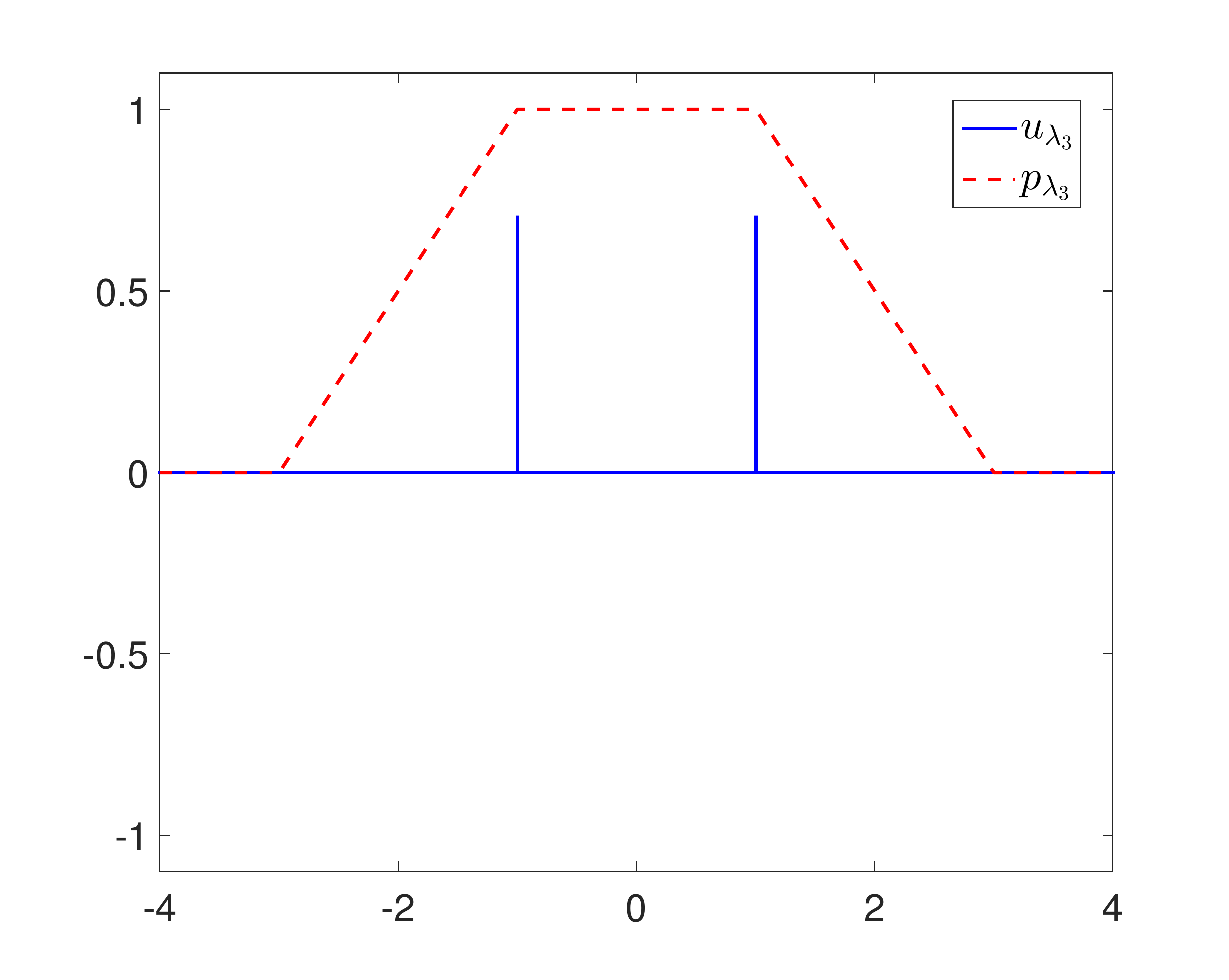}\label{subfig: l1conv2_example_singvecs_3}}\\
\subfloat[\textcolor{blue}{$u_{\lambda_4}$} and \textcolor{red}{$p_{\lambda_4}$}.]{\includegraphics[width=0.32\textwidth]{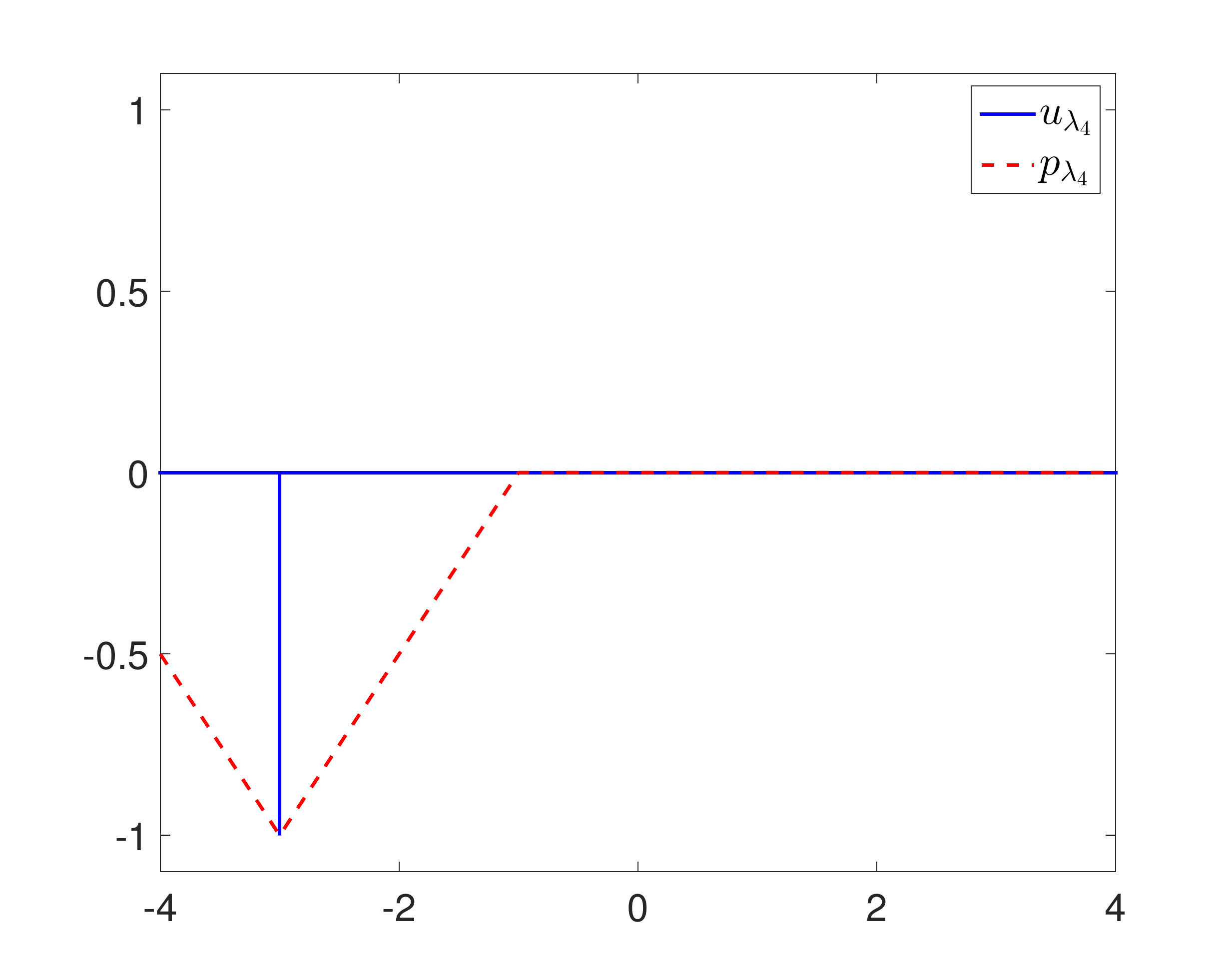}\label{subfig: l1conv2_example_singvecs_4}}
\subfloat[\textcolor{blue}{$u_{\lambda_5}$} and \textcolor{red}{$p_{\lambda_5}$}.]{\includegraphics[width=0.32\textwidth]{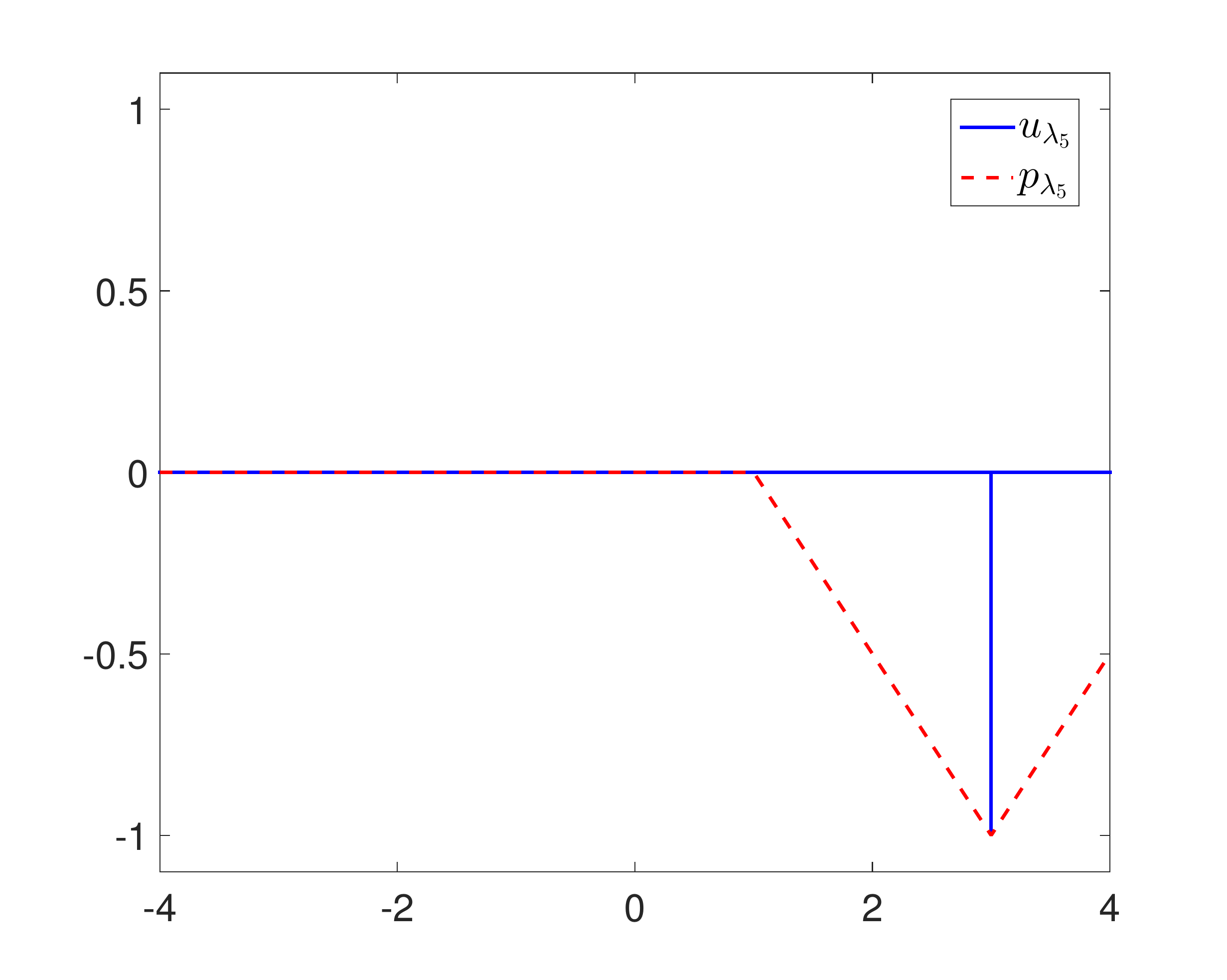}\label{subfig: l1conv2_example_singvecs_5}}
\end{center}
\vspace{0.5cm}
\caption{Singular vectors and their corresponding subgradients for Example \ref{example: PartialSUB0Necessary}}
\label{fig: l1conv2_example_singvecs}
\end{figure}

We note that $\ip{Ku_{\lambda_i}}{Ku_{\lambda_j}} = \frac{1}{\lambda_i}\ip{p_{\lambda_i}}{u_{\lambda_j}} = 0$ for $i\neq j$. Hence the orthogonality condition (\ref{eq: orthogonality_cond}) is satisfied. The singular vectors are also normalised such that $\norm{Ku_{\lambda_j}}_{\ell^2(\R^m)} = 1$ for $j=1,...,5$ holds true. Now let $\{\gamma_j\}_{j=1}^5$ be such that $\lambda_k/\gamma_k < \lambda_{k+1}/\gamma_{k+1}$ for $k=1,...,4$. Then if the (\ref{eq: subdiff_cond}) condition is satisfied, Theorem \ref{thm: ExactReconstruction} tells us that the ISS flow gives a decomposition of the data $f = \sum_{j=1}^{5} \gamma_jKu_{\lambda_j}$ into the singular vectors. But the (\ref{eq: subdiff_cond}) condition is not satisfied since the partial sums over the subgradients are not in the subdifferential of $J$ at zero:
\begin{align*}
p_{\lambda_1}+p_{\lambda_2} &= (0,0,1,0,\textcolor{red}{-2},0,1,0,0)^T \notin\partial J(0), \\
p_{\lambda_1}+p_{\lambda_2}+p_{\lambda_3} &= (0,0,\textcolor{red}{3/2},1,-1,1,\textcolor{red}{3/2},0,0)^T \notin\partial J(0), \\
p_{\lambda_1}+p_{\lambda_2}+p_{\lambda_3}+p_{\lambda_4} &= (-1/2,-1,1,1,-1,1,\textcolor{red}{3/2},0,0)^T\notin \partial J(0), \\
p_{\lambda_1}+p_{\lambda_2}+p_{\lambda_3}+p_{\lambda_4}+p_{\lambda_5} &= (-1/2,-1,1,1,-1,1,1,-1,-1/2)^T \in \partial J(0).
\end{align*}
Only the full sum over all the subgradients is in the subdifferential of $J$ at zero. In Figure \ref{fig: l1conv2_example_sum_p} we see some of the partial sums over the subgradients. In Figure \ref{subfig: l1conv2_example_sum3} we see the sum over the first three subgradients. We see that the sum does not have its maximum magnitude at the peaks and it exceeds the limit of a magnitude of one. Hence the sum of the first three subgradients is not in the subdifferential of $J$ at zero. In Figure \ref{subfig: l1conv2_example_sum5} we see the sum over all five subgradients. This sum has its maximum magnitude at the peaks and does not exceed the magnitude limit of one. Hence the full sum is in the subdifferential of $J$ at zero. \\

\begin{figure}[!h]
\begin{center}
\subfloat[\textcolor{blue}{$\sum_{j=1}^{3} u_{\lambda_j}$} and \textcolor{red}{$\sum_{j=1}^{3} p_{\lambda_j}$}.]{\includegraphics[width=0.49\textwidth]{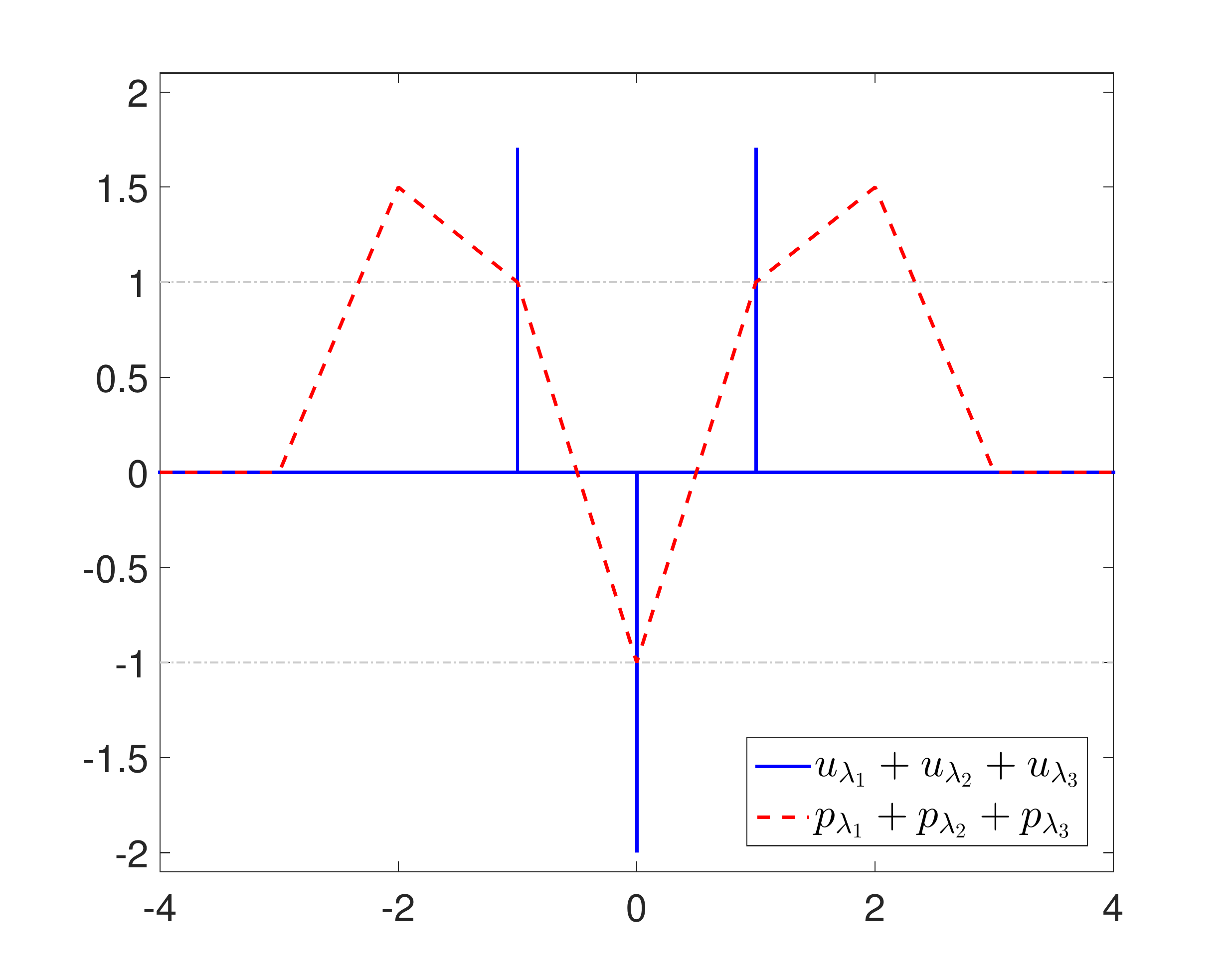}\label{subfig: l1conv2_example_sum3}}
\subfloat[\textcolor{blue}{$\sum_{j=1}^{5} u_{\lambda_j}$} and \textcolor{red}{$\sum_{j=1}^{5} p_{\lambda_j}$}]{\includegraphics[width=0.49\textwidth]{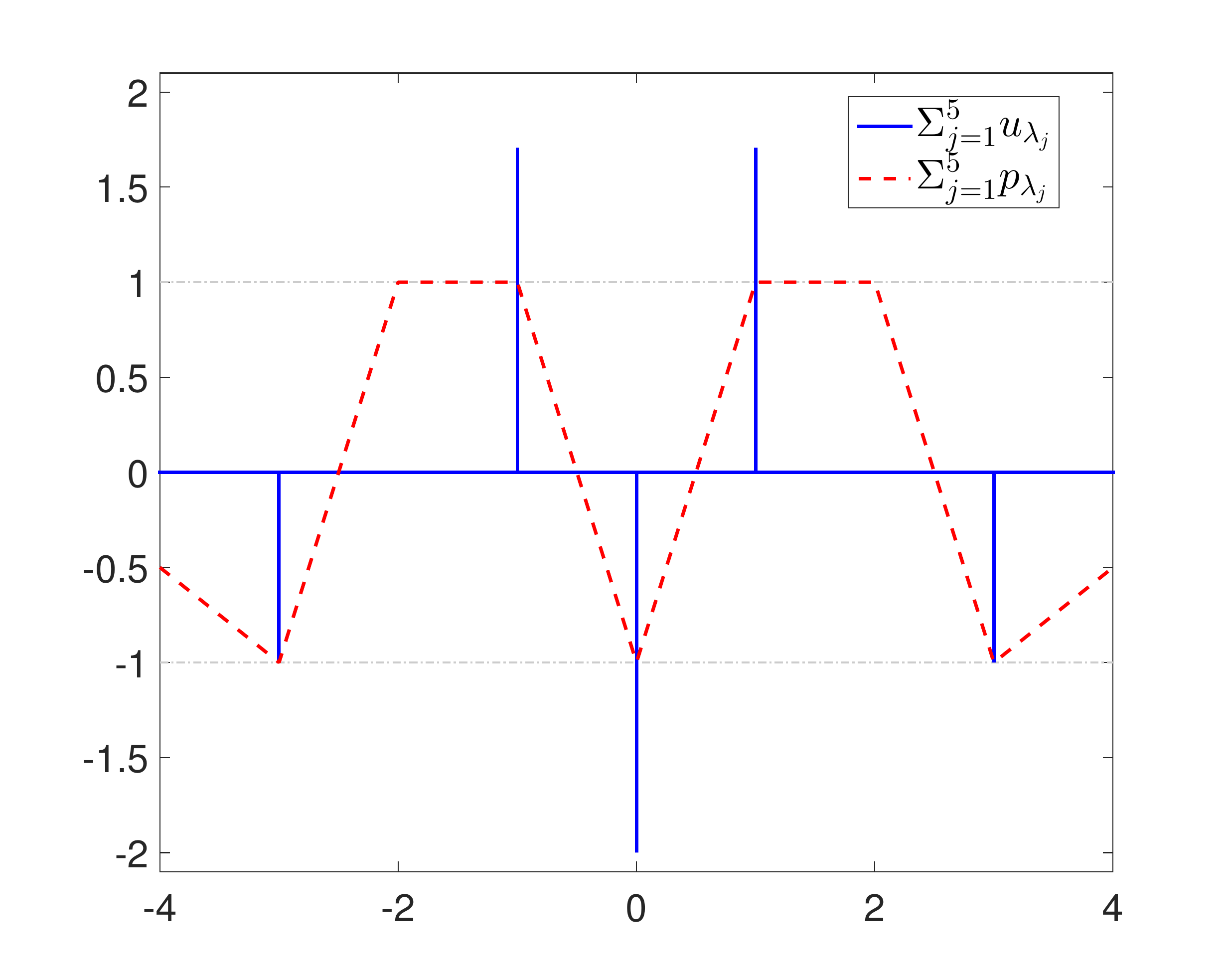}\label{subfig: l1conv2_example_sum5}}
\end{center}
\vspace{0.5cm}
\caption{Sum of subgradients for Example \ref{example: PartialSUB0Necessary}. We see that the partial sum conditions of the (\ref{eq: subdiff_cond}) condition are not satisfied. The sum over the first three subgradients seen in Figure \ref{subfig: l1conv2_example_sum3} exceeds one and can therefore not be in the subdifferential of $J$ at zero. In Figure \ref{subfig: l1conv2_example_sum5} the full sum over the subgradients is seen and this is in the subdifferential of $J$ at zero.}
\label{fig: l1conv2_example_sum_p}
\end{figure}

We are going to see in Section \ref{sec: SUB0violations} that the first peak reconstructed by the ISS flow in this case is at the middle position. This peak does not match any of the singular vectors $u_{\lambda_j}$, $j=1,\dots,5$, and we will not obtain the decomposition.

\end{example}

\subsection{Singular Vector Fusion}
One might think that if two or more singular vectors representing the data $f$ add up to another singular vector, then the ISS flow is not able to give a complete decomposition into the singular vectors. However, it turns out that under the (\ref{eq: subdiff_cond}) condition and the orthogonality condition (\ref{eq: orthogonality_cond}), two or more singular vectors cannot add up to another singular vector. \\

\begin{proposition}[Impossible Singular Vector Fusion] 
\label{prop: SingularVectorFusion}
Let $\{u_{\lambda_j}\}_{j=1}^{n}$ be a set of $K$-normalised singular vectors of $J$ with corresponding singular values $\{\lambda_j\}_{j=1}^{n}$. Assume that they satisfy the orthogonality condition (\ref{eq: orthogonality_cond}) and the (\ref{eq: subdiff_cond}) condition. Let the data $f$ be given by $f = \sum_{j=1}^{n} \gamma_j Ku_{\lambda_j}$, where $\gamma_j>0$ and $\lambda_j/\gamma_j < \lambda_{j+1}/\gamma_{j+1}$. Let $\{u_{\lambda_{j_k}}\}_{k=1}^m$, $m\in\{2;n\}$, be a subset of the singular vectors ordered so that $\lambda_{j_k}/\gamma_{j_k}<\lambda_{j_{k+1}}/\gamma_{j_{k+1}}$. Then no such subset of singular vectors can add up to another singular vector, i.e.
\begin{align*}
\frac{J\left(\sum_{k=1}^{m} \gamma_{j_k}u_{\lambda_{j_k}}\right)}{\norm{ K\left(\sum_{k=1}^{m} \gamma_{j_k} u_{\lambda_{j_k}}\right)}_{\hc}^{2}}K^*K\left(\sum_{k=1}^{m} \gamma_{j_k}u_{\lambda_{j_k}}\right) \notin \partial J\left(\sum_{k=1}^{m} \gamma_{j_k} u_{\lambda_{j_k}}\right).
\end{align*}
\end{proposition}
\begin{proof}
The proof is by contradiction. Hence we assume that $\sum_{k=1}^{m} \gamma_{j_k}u_{\lambda_{j_k}}$ is a singular vector. Using (\ref{eq: orthogonality_cond}) and Proposition \ref{prop: J_linearity} we can explicitly compute the singular value:
\begin{align*}
\frac{J\left(\sum_{k=1}^{m} \gamma_{j_k}u_{\lambda_{j_k}}\right)}{\norm{ K\left(\sum_{k=1}^{m} \gamma_{j_k} u_{\lambda_{j_k}}\right)}_{\hc}^{2}} = \frac{\sum_{k=1}^{m}\gamma_{j_k}\lambda_{j_k}}{\sum_{k=1}^{m} \gamma_{j_k}^2}.
\end{align*}
Using the characterisation (\ref{eq: subdiff_characterization}) of the subdifferential we then get
\begin{align*}
J(v) &\geq \ip{\frac{\sum_{k=1}^{m}\gamma_{j_k}\lambda_{j_k}}{\sum_{k=1}^{m} \gamma_{j_k}^2}K^*K\left(\sum_{k=1}^{m} \gamma_{j_k}u_{\lambda_{j_k}}\right)}{v} \\
&= \frac{\sum_{k=1}^{m}\gamma_{j_k}\lambda_{j_k}}{\sum_{k=1}^{m} \gamma_{j_k}^2}\sum_{k=1}^{m}\ip{\gamma_{j_k}K^*Ku_{\lambda_{j_k}}}{v}.
\end{align*}
for all $v\in\dom(J)$. Set $v=u_{\lambda_{j_1}}$. Using $J(u_{{\lambda_j}_1}) = \lambda_{j_1}$ and (\ref{eq: orthogonality_cond}) in the above we obtain
\begin{align*}
\lambda_{j_1} \geq \frac{\sum_{k=1}^{m}\gamma_{j_k}\lambda_{j_k}}{\sum_{k=1}^{m} \gamma_{j_k}^2}\gamma_{j_1}.
\end{align*}
Rearranging and writing out the first term of both sums, the above implies
\begin{align*}
\lambda_{j_1}\gamma_{j_1}^2 + \lambda_{j_1}\sum_{k=2}^{m} \gamma_{j_k}^2 = \lambda_{j_1}\sum_{k=1}^{m} \gamma_{j_k}^2 \geq \gamma_{j_1}\sum_{k=1}^{m} \gamma_{j_k}\lambda_{j_k} = \gamma_{j_1}^2\lambda_{j_1} + \gamma_{j_1}\sum_{k=2}^{m} \gamma_{j_k}\lambda_{j_k}.
\end{align*}
We let the term $\lambda_{j_1}\gamma_{j_1}^2$ on both sides cancel and rearrange again to obtain
\begin{align*}
\frac{\lambda_{j_1}}{\gamma_{j_1}} \geq \frac{\sum_{k=2}^{m} \gamma_{j_k}\lambda_{j_k}}{\sum_{k=2}^{m} \gamma_{j_k}^2} = \frac{\gamma_{j_2}\lambda_{j_2}+\sum_{k=3}^{m} \gamma_{j_k}\lambda_{j_k}}{\sum_{k=2}^{m} \gamma_{j_k}^2}.
\end{align*}
Using the ordering $\lambda_{j_k}/\gamma_{j_k} < \lambda_{j_{k+1}}/\gamma_{j_{k+1}}$ then for $k\geq 3$ we get $\lambda_{j_k} > \frac{\lambda_{j_2}}{\gamma_{j_2}}\gamma_{j_k}$. Applying this in the above we reach
\begin{align*}
\frac{\lambda_{j_1}}{\gamma_{j_1}} > \frac{\gamma_{j_2}\lambda_{j_2} + \frac{\lambda_{j_2}}{\gamma_{j_2}}\sum_{k=3}^{m}\gamma_{j_k}^2}{\sum_{k=2}^{m} \gamma_{j_k}^2} = \frac{\frac{\lambda_{j_2}}{\gamma_{j_2}}\left(\gamma_{j_2}^2+\sum_{k=3}^{m} \gamma_{j_k}^2\right)}{\sum_{k=2}^{m} \gamma_{j_k}^2} = \frac{\lambda_{j_2}}{\gamma_{j_2}}.
\end{align*}
This is a contradiction since by assumption $\lambda_{j_1}/\gamma_{j_1} < \lambda_{j_2}/\gamma_{j_2}$. Hence the assumption that $\sum_{k=1}^{m} \gamma_{j_k}u_{\lambda_{j_k}}$ is a singular vector is false.
\end{proof}

\section{Decomposition of Arbitrary Input Data}
Having addressed the question of when a finite sum of singular vectors can be decomposed via the ISS flow (\ref{eq: ISSflow}), we want to draw our attention to a rather converse question: when is the first non-trivial solution of (\ref{eq: ISSflow}) a singular vector? We first give a characterisation of the first non-trivial solution of the ISS flow when the data $f$ is given by $f = K\omega$, where $\omega$ satisfies the strong source condition (\ref{eq:ssc}). Then we give a characterisation of the first non-trivial ISS solution for arbitrary data $f\in\hc$ and observe that it satisfies the source condition (\ref{eq:sc}). Finally, we investigate under which conditions this first non-trivial solution is a singular vector. Throughout the section Assumption \ref{assumption: overall} will be used.

\subsection{Strong Source Condition}
We identify the first step of the ISS flow for given data $f = K\omega$, where $\omega\in\uc\setminus\{0\}$ is the source element for the strong source condition (\ref{eq:ssc}), i.e. $K^*K\omega\in\partial J(u^\dagger)$ for some $u^{\dagger}\in\dom(J)$. The first non-trivial solution of the ISS flow under this assumption for the data is then characterised via the following lemma.
\begin{lemma}
\label{lemma: SSCsolution}
Let (\ref{eq:ssc}) be satisfied, i.e. there exists a source element $\omega\in\uc\setminus\{0\}$ such that $K^*K\omega\in\partial J(u^{\dagger})$. Let $f=K\omega$ be the given data for the ISS flow (\ref{eq: ISSflow}). Then a first non-trivial solution of the flow is given by
\begin{align*}
u(t_1) = \frac{J(u^{\dagger})}{\norm{Ku^{\dagger}}_{\hc}^{2}}u^{\dagger},\; t_1 = 1.
\end{align*}
\end{lemma}
\begin{proof}
Starting with $u(t) = 0$ for $0\leq t < t_1$ we need to compute $t_1$ and $u(t_1)$. By the ISS flow we get
\begin{align*}
\partial_t p(t) = K^*f = K^*K\omega,\; 0\leq t < t_1.
\end{align*}
Using $p(0) = 0$ and integrating $\int_{0}^{t} \partial_\tau p(\tau) d\tau$ we get
\begin{align*}
p(t) = tK^*K\omega, \; 0\leq t < t_1.
\end{align*}
We need to show that $p(t)\in\partial J(u(t)) = \partial J(0)$ for $0\leq t < t_1$.  Using the strong source condition and the characterisation (\ref{eq: subdiff_characterization}) of the subdifferential, we obtain
\begin{align*}
\ip{p(t)}{v} = t\ip{K^*K\omega}{v} \leq tJ(v) < t_1J(v), \quad \forall v\in\uc.
\end{align*}
Hence choosing $t_1 = 1$ we know that $p(t)\in\partial J(0)$ for $0\leq t<t_1$. At the time $t_1$ we require $p(t_1) = t_1K^*K\omega = K^*K\omega\in\partial J(u(t_1))$. Hence choosing $u(t_1) = \gamma u^{\dagger}$ for $\gamma \geq 0$ the condition is satisfied by the positive scaling invariance (\ref{eq: subdiff_scaling_invariant}) of the subdifferential. We use the necessary orthogonality condition (\ref{eq: general_orthogonality_cond}) to determine $\gamma$:
\begin{align*}
0 = \ip{K^*(K\omega-\gamma Ku^{\dagger})}{u^{\dagger}} =  \ip{K^*K\omega}{u^{\dagger}} - \gamma\norm{Ku^{\dagger}}_{\hc}^2 = J(u^{\dagger})-\gamma\norm{Ku^{\dagger}}_{\hc}^2,
\end{align*}
i.e. $\gamma = J(u^{\dagger})/\norm{Ku^{\dagger}}_{\hc}^2 \geq 0$. Hence a first non-trivial step of the ISS flow is
\begin{align*}
u(t_1) = \frac{J(u^{\dagger})}{\norm{Ku^{\dagger}}_{\hc}^{2}}u^{\dagger},\quad t_1 = 1.
\end{align*}
\end{proof}

\subsection{Source Condition}
We emphasise the connection between the first non-trivial solution of the ISS flow (\ref{eq: ISSflow}) and the source condition \eqref{eq:sc} for arbitrary data $f$. It turns out that the first solution satisfies the source condition, and under the special kind of source condition \eqref{eq:noisecond} the first non-trivial solution is also a singular vector. \\

We show that the first non-trivial solution of the ISS flow satisfies the source condition (\ref{eq:sc}). Since $u(t) = 0$ for $0\leq t < t_1$ for some $t_1 > 0$, using the ISS flow we get
\begin{align*}
\partial_t p(t) = K^*f,\; 0\leq t < t_1.
\end{align*}
Using $p(0) = 0$ and integration yields
\begin{align*}
p(t) = tK^*f,\; 0\leq t < t_1.
\end{align*}
At $t=t_1$ we then require
\begin{align*}
p(t_1) = t_1K^*f \in \partial J(u(t_1))
\end{align*}
by the definition of the ISS flow. Setting $v=t_1f\in\hc$ we obtain $K^*v\in\partial J(u(t_1))$ which is exactly the source condition. Hence $u(t_1)$ satisfies the source condition. \\ 

In Theorem \ref{thm:noisyiss} it has been shown that for data given as $f = \gamma Ku_\lambda + g$, for some $g \in \hc$, the first non-trivial solution of the ISS solution at time $t_1 = \lambda\eta/(\lambda + \gamma \eta - \mu) < t_2 = \eta$ is given via $u(t_1) = (\gamma + (\lambda - \mu)/\eta) u_\lambda$ if the condition \eqref{eq:noisecond} for some $\gamma > \mu/\eta$ is met. We also want to emphasise that \eqref{eq:noisecond} is in fact nothing but the source condition \eqref{eq:sc} for the particular choice $v = \mu Ku_\lambda + \eta g$. So we do know from Theorem \ref{thm:noisyiss} that under this specific kind of source condition the first non-trivial ISS solution is indeed a generalised singular vector. This can be of practical interest if we have arbitrary data $f$ that can be decomposed into a singular vector and an element $g$ such that \eqref{eq:noisecond} is satisfied. However, one has to be very careful that the conditions for a successful application of Theorem \ref{thm:noisyiss} are actually met, as the following example shows.
\begin{example}
Let us consider the embedding operator $K = I:\text{BV}([0, 1]) \rightarrow L^2([0, 1])$ and the one-dimensional total variation regularisation 
\begin{align*}
J(u) = \tv(u) := \sup_{\substack{\varphi \in C_0^\infty([0, 1])\\ \| \varphi \|_{L^\infty([0, 1])} \leq 1}} \int_0^1 u(x) \varphi^\prime(x) \, dx \, \text{.}
\end{align*}
Given the function $f(x) = -\cos(2\pi x)$ it is quite obvious that we can decompose $f$ into $f = \gamma u_\lambda + g$ for a constant $\gamma>0$, 
\begin{align*}
u_\lambda(x) = \begin{cases} 1, & x \in \left[\frac{1}{4}, \frac{3}{4}\right]\\ -1, & \text{else} \end{cases} \quad \text{and} \quad g(x) = \begin{cases} -\gamma - \cos(2\pi x), & x \in \left[\frac{1}{4}, \frac{3}{4}\right]\\ \gamma - \cos(2\pi x), & \text{else} \end{cases} \, .
\end{align*}
Hence, if we could verify \eqref{eq:noisecond} for two constants $\mu$ and $\eta$ with $\mu/\eta < \gamma$, we would know from Theorem \ref{thm:noisyiss} that $(\gamma + (\lambda-\mu)/\eta)u_\lambda$ would be the first non-trivial solution of the ISS flow. In order to do so we have to characterise the subdifferential of $\tv$, which reads as
\begin{align*}
\partial \tv(u) = \left\{ q^\prime \ | \ \| q \|_{L^\infty([0, 1])} \leq 1, \,  q(0) = q(1) = 0, \ \text{and} \  \langle q^\prime,u \rangle = \tv(u) \right\} \, .
\end{align*}
In our case we therefore have to find constants $\mu$ and $\eta$ and a function $q$ with $q^\prime = \mu u_\lambda + \eta g$, $q(0) = q(1) = 0$, $\| q \|_{L^\infty([0, 1])} \leq 1$ and $\langle q^\prime,\ul \rangle = \tv(\ul)$. We make the assumption $\eta = 2\pi$ and verify that we can find a function $q$ that satisfies all conditions listed above for some $\mu$. We do so by considering $q(x) := -\sin(2\pi x)$. We immediately observe $q(0) = q(1) = 0$ and $\| q \|_{L^\infty([0, 1])} = 1$. For the choice $\eta = 2\pi$ we further compute
\begin{align*}
q^\prime(x) = -2\pi \cos(2\pi x) = \mu \ul(x) + \eta \begin{cases} -\frac{\mu}{\eta} - \cos(2\pi x), & x \in \left[\frac{1}{4}, \frac{3}{4}\right]\\ \frac{\mu}{\eta} - \cos(2\pi x), & \text{else}\end{cases} \, .
\end{align*}
For the choice $\mu = \gamma\eta$ we then obtain $-q^\prime = \mu \ul + \eta w$, and further observe $\langle q^\prime, \ul\rangle = 4 = \tv(\ul)$. Hence, we have verified $\mu \ul + \eta w \in \partial \tv(\ul)$ for $\eta = 2\pi$ and $\mu = \gamma\eta$. Unfortunately, we do not obtain $\mu/\eta < \gamma$, which means that we cannot apply Theorem \ref{thm:noisyiss} to automatically conclude that the first non-trivial solution of the ISS flow is given via a singular vector, though it seems tempting at first glance. 
\end{example}

\subsection{Arbitrary input data}
In this section we want to characterise the first non-trivial solution of the ISS flow for arbitrary given data. We investigate under which conditions the first solution is guaranteed to be a singular vector. This investigation leads to the definition of dual singular vectors in the data space. \\

Under the norm-inequality assumption (\ref{eq: NormIneqAssumption}) for $J$ we can characterise the first non-trivial solution of the ISS flow (\ref{eq: ISSflow}) via the following lemma.
\begin{lemma}[First Non-trivial Solution of ISS Flow] 
\label{lemma: FirstNonTrivialSol}
Assume that $J$ satisfies the norm-inequality assumption (\ref{eq: NormIneqAssumption}). Let $f\in\hc\setminus \ker(K^*)$ be arbitrary given data for the ISS flow (\ref{eq: ISSflow}) satisfying $\ip{K^*f}{v} = 0$ for all $v\in \ker(J)$. Then the first non-trivial solution of the flow is characterised by
\begin{align*}
u(t_1) = c_1u_1
\end{align*}
where $t_1=c_0/\norm{K^*f}_{\uc^*}$ ($c_0$ being the norm-inequality constant), $c_1 = J(u_1)/t_1$ and $u_1 = v_1/\norm{Kv_1}_{\hc}$ with
\begin{align}
v_1\in\argmin_{v\in\ker(J)^{\perp}} \left\{J(v)-t_1\ip{K^*f}{v}\right\}.
\label{eq: FirstISSsol}
\end{align}
\end{lemma}
\begin{proof}
For $0\leq t < t_1$ we have $u(t) = 0$. The ISS flow then gives
\begin{align*}
\partial_t p(t) = K^*f,\; 0\leq t < t_1.
\end{align*}
Using $p(0) = 0$ and integration we obtain $p(t) = tK^*f$ for $0\leq t < t_1$. This implies
\begin{align*}
\norm{p(t)}_{\uc^*} = t\norm{K^*f}_{\uc^*} < t_1\norm{K^*f}_{\uc^*} = c_0.
\end{align*}
Together with the assumption that $\ip{K^*f}{v} = 0$ for all $v\in\ker(J)$, Proposition \ref{prop: DualNormFormOfSubdiff} then implies $p(t)\in\partial J(0)$. Hence the ISS flow is satisfied for $0\leq t < t_1$. At $t = t_1$ we need $p(t_1) = t_1K^*f\in\partial J(c_1u_1) = \partial J(u_1) = \partial J(v_1)$ for $c_1u_1$ to be the solution at time $t_1$. This is exactly the optimality condition of
\begin{align*}
v_1\in\argmin_{v\in\ker(J)^{\perp}} \left\{J(v)-t_1\ip{K^*f}{v}\right\},
\end{align*}
since the Lagrange multiplier for the constraint $v\in\ker(J)^{\perp}$ is zero following the same procedure as of that for the generalised singular vectors in Section \ref{sec: Preliminaries}. Finally, we require $u(t_1)$ to satisfy the necessary orthogonality condition (\ref{eq: general_orthogonality_cond}):
\begin{align*}
0 = \ip{K^*(f-Ku(t_1))}{u(t_1)} = c_1\ip{K^*f}{u_1} - c_1^2 = c_1\left(\frac{J(u_1)}{t_1} -c_1\right).
\end{align*}
We have used that $t_1K^*f\in\partial J(u_1)$ for the last equality. From the above we obtain
\begin{align*}
c_1 = \frac{J(u_1)}{t_1}.
\end{align*}
\end{proof}
The interesting question now is when the first ISS solution $u_1$ characterised via (\ref{eq: FirstISSsol}) is a singular vector. We need to investigate if $J(u_1)K^*Ku_1\in\partial J(u_1)$. Using the characterisation of the subdifferential (\ref{eq: subdiff_characterization}) we need to show
\begin{align*}
\ip{J(u_1)K^*Ku_1}{u_1} = J(u_1) \quad\text{and}\quad J(u_1)K^*Ku_1\in\partial J(0).
\end{align*}
The first condition is easy to verify since $\norm{Ku_1}_{\hc} = 1$. For the second condition we use Proposition \ref{prop: DualNormFormOfSubdiff}, thus we need $\norm{J(u_1)K^*Ku_1}_{\uc^*} \leq c_0$. For $J(u_1)$ we use the estimate
\begin{align*}
J(u_1) = t_1\ip{K^*f}{u_1} = \frac{c_0}{\norm{K^*f}_{\uc^*}}\ip{f}{Ku_1} \leq \frac{c_0}{\norm{K^*f}_{\uc^*}}\norm{f}_{\hc}\norm{Ku_1}_{\hc} = c_0\frac{\norm{f}_{\hc}}{\norm{K^*f}_{\uc^*}}.
\end{align*}
Then we obtain
\begin{align*}
\norm{J(u_1)K^*Ku_1}_{\uc^*} \leq c_0\frac{\norm{f}_{\hc}}{\norm{K^*f}_{\uc^*}}\norm{K^*}\norm{Ku_1}_{\hc} = c_0\frac{\norm{f}_{\hc}\norm{K^*}}{\norm{K^*f}_{\uc^*}},
\end{align*}
where $\norm{K^*}$ is the standard operator norm
\begin{align*}
\norm{K^*} = \sup_{\omega\in\hc\setminus\{0\}} \frac{\norm{K^*\omega}_{\uc^*}}{\norm{\omega}_{\hc}}.
\end{align*} 
Hence we see that if $\norm{K^*f}_{\uc^*} = \norm{K^*}\norm{f}_{\hc}$ then the first ISS solution $u_1$ is actually a singular vector of $J$. We note here that one should be very careful using the correct norms. To emphasise this we consider the following example:
\begin{example}
\label{example: DualSingVec}
Let $K = I: H^1_0(\Omega) \to L^2_0(\Omega)$ be an imbedding operator and $\Omega$ a bounded domain in $\R^2$. Here subscript zero means that the functions are zero on the boundary of $\Omega$. Let $J(u) = \norm{\nabla u}_{L^2(\Omega)}$. Using Poincar\'e's inequality $\norm{u}_{L^2(\Omega)} \leq C\norm{\nabla u}_{L^2(\Omega)}$ for some $C>0$ and $u\in\ker(J)^{\perp}$, we can verify that $J$ satisfies the norm-inequality assumption (\ref{eq: NormIneqAssumption}) for $c_0 = 1/(1+C^2)^{1/2}$. Let $f\in H^1(\Omega)$ be the data for the ISS flow. If $f$ satisfies
\begin{align*}
\norm{I^*f}_{H^1(\Omega)} = \norm{I^*}\norm{f}_{L^2(\Omega)} = \norm{f}_{L^2(\Omega)},
\end{align*}
the first non-trivial solution of the ISS flow is a singular vector. The equality is not always satisfied even though $K$ is an embedding operator. We see that we need
\begin{align*}
\left(\norm{f}^2_{L^2(\Omega)}+\norm{\nabla f}^2_{L^2(\Omega)}\right)^{1/2} = \norm{f}_{L^2(\Omega)} \quad \Leftrightarrow\quad \norm{\nabla f}_{L^2(\Omega)} = 0.
\end{align*}
\end{example}
The condition $\norm{K^*f}_{\uc^*} = \norm{K^*}\norm{f}_{\hc}$ gives rise to the following theorem:
\begin{theorem}
\label{thm: FirstISSsolIsSingVec}
Assume that $J$ satisfies the norm-inequality assumption (\ref{eq: NormIneqAssumption}). If $f\in\hc\setminus\ker(K^*)$ is a \emph{dual singular vector}, i.e.
\begin{align}
\frac{\norm{K^*f}_{\uc^*}}{\norm{f}_{\hc}^2}f\in\partial\norm{K^*f}_{\uc^*},
\label{eq: dualSV}
\end{align}
and satisfies $\ip{K^*f}{v} = 0 $ for all $v\in\ker(J)$, then the first non-trivial solution $u_1$ of the ISS flow (\ref{eq: ISSflow}) characterised via (\ref{eq: FirstISSsol}) is a (primal) singular vector of $J$.
\end{theorem}
\begin{proof}
We have shown above that if $f$ satisfies $\norm{K^*f}_{\uc^*} = \norm{K^*}\norm{f}_{\hc}$, then $u_1$ is a singular vector. Since
\begin{align*}
\norm{K^*} = \sup_{\omega\in\hc} \frac{\norm{K^*\omega}_{\uc^*}}{\norm{\omega}_{\hc}}
\end{align*}
we see that in order for $\norm{K^*f}_{\uc^*} = \norm{K^*}\norm{f}_{\hc}$ to hold true, $f$ has to satisfy
\begin{align*}
f\in\argmax_{\omega\in\hc} \frac{\norm{K^*\omega}_{\uc^*}}{\norm{\omega}_{\hc}}.
\end{align*}
The optimality condition for this is
\begin{align*}
0\in \frac{\partial\norm{K^*f}_{\uc^*}\norm{f}_{\hc}-\norm{K^*f}_{\uc^*}\frac{f}{\norm{f}_{\hc}}}{\norm{f}_{\hc}^2}
\end{align*}
which implies
\begin{align*}
\frac{\norm{K^*f}_{\uc^*}}{\norm{f}_{\hc}^2}f\in\partial\norm{K^*f}_{\uc^*}.
\end{align*}
\end{proof}
\begin{remark}
Setting $J_2(\omega) = \norm{K^*\omega}_{\uc^*}$, $\omega\in\hc$, we see from the above that $f$ has to be a singular vector of $J_2$ (and $I:\hc\to\hc$). $J_2$ is absolutely one-homogeneous and hence we can reuse all the results for absolutely one-homogeneous functionals.
\end{remark}




\section{Numerical Results}\label{sec:results}
We conclude the paper with numerical results that demonstrate the sufficiency of the \eqref{eq: subdiff_cond} condition when we wish to decompose a linear combination of generalised singular vectors. In order to verify the theoretical results provided in Section \ref{sec:decomp}, we need to be able to solve the ISS flow (\ref{eq: ISSflow}) numerically. In \cite{burger2013adaptive} it has been shown that solutions of (\ref{eq: ISSflow}) can be computed directly in case of $J(u) = \| u  \|_{\ell^1(\R^m)}$ via what has been termed the adaptive inverse scale space method (aISS). In \cite{moeller2013multiscale} the aISS has been extended to general polyhedral functionals $J$ and could therefore also be used to numerically verify the behaviour of functionals like the anisotropic total variation, for example. For arbitrary convex, absolutely one-homogeneous functionals $J$, however, there does not exist a way to efficiently compute the exact ISS solution at each time step numerically to the best of our knowledge. Therefore, for simplicity, we are going to restrict ourselves to the case of $J(u) = \| u \|_{\ell^1(\R^n)}$ for the numerical results. We compute the inverse scale space solution for a composition satisfying the (\ref{eq: subdiff_cond}) condition. We also demonstrate the need for the (\ref{eq: subdiff_cond}) condition computing ISS solutions of compositions following Examples \ref{example: l1_convolution} and \ref{example: PartialSUB0Necessary}.


\subsection{\eqref{eq: subdiff_cond} Observance}
\label{sec: SUB0observance}
With our first example we simply verify that the result of Theorem \ref{thm: ExactReconstruction} is observed in practice. In order to do so, we consider the operator $K:\ell^1(\R^8)\supseteq\ell^2(\R^8) \to \ell^2(\R^8)$ defined by $(Ku)_k = \sum_{j=-1}^{1} u_{k-j}g_j$ for $g = (1,1,0)^T / \sqrt{2}$ from Example \ref{example: l1_convolution}, respectively Example \ref{example: PartialSUB0Necessary}, the absolutely one-homogeneous functional $J(u) = \norm{u}_{\ell^1(\R^8)}$, and the following three $K$-normalised and $K$-orthogonal singular vectors
\begin{align*}
u_{\lambda_1} &= (0,0,0,1,-1,0,0,0)^T, \\
u_{\lambda_2} &= (0,-1,0,0,0,0,0,0)^T, \\ 
u_{\lambda_3} &= (0,0,0,0,0,0,1,0)^T,
\end{align*}
with $\lambda_1 = 2$ and $\lambda_2 = \lambda_3 = 1$. We now verify that these singular vectors satisfy the \eqref{eq: subdiff_cond} condition in the order they are defined, i.e. we have to verify $\lambda_1 K^* K u_{\lambda_1} \in \partial J(0)$, $\lambda_1 K^* K u_{\lambda_1} + \lambda_2 K^* K u_{\lambda_2} \in \partial J(0)$ and $\lambda_1 K^* K u_{\lambda_1} + \lambda_2 K^* K u_{\lambda_2} + \lambda_3 K^* K u_{\lambda_3} \in \partial J(0)$. The first condition trivially follows from $u_{\lambda_1}$ being a generalised singular vector, the second and third follow from
\begin{align*}
\sum_{j = 1}^2 \lambda_j K^* K u_{\lambda_j} &= (-1/2,-1,1/2,1,-1,-1,0,0)^T \in \partial J(0), \\
\sum_{j = 1}^3 \lambda_j K^* K u_{\lambda_j} &= (-1/2,-1,1/2,1,-1,-1/2,1,1/2)^T \in \partial J(0).
\end{align*}
Hence, \eqref{eq: subdiff_cond} is satisfied and Theorem \ref{thm: ExactReconstruction} guarantees that \eqref{eq: ISSflow} and therefore the aISS algorithm will decompose any composition $f = \sum_{j = 1}^3 \gamma_j Ku_{\lambda_j}$ with $\lambda_1/\gamma_1 < \lambda_2/\gamma_2 < \lambda_3 / \gamma_3$ into $\gamma_1 u_{\lambda_1}$, $\gamma_1 u_{\lambda_1} + \gamma_2 u_{\lambda_2}$ and $\gamma_1 u_{\lambda_1} + \gamma_2 u_{\lambda_2} + \gamma_3 u_{\lambda_3}$. Figure \ref{fig:positivesub0exm} shows the three iterations of the aISS method until convergence, for $f = \sum_{j = 1}^3 \gamma_j Ku_{\lambda_j}$ with $\gamma_1 = 5$, $\gamma_2 = 2$ and $\gamma_3 = 1$. It can be seen that the aISS algorithm and therefore \eqref{eq: ISSflow} indeed decomposes the signal $f$ as stated by Theorem \ref{thm: ExactReconstruction}.
\begin{figure}[!ht]
       \centering
       \subfloat[First aISS solution and subgradient.]{\includegraphics[width=0.49\textwidth]{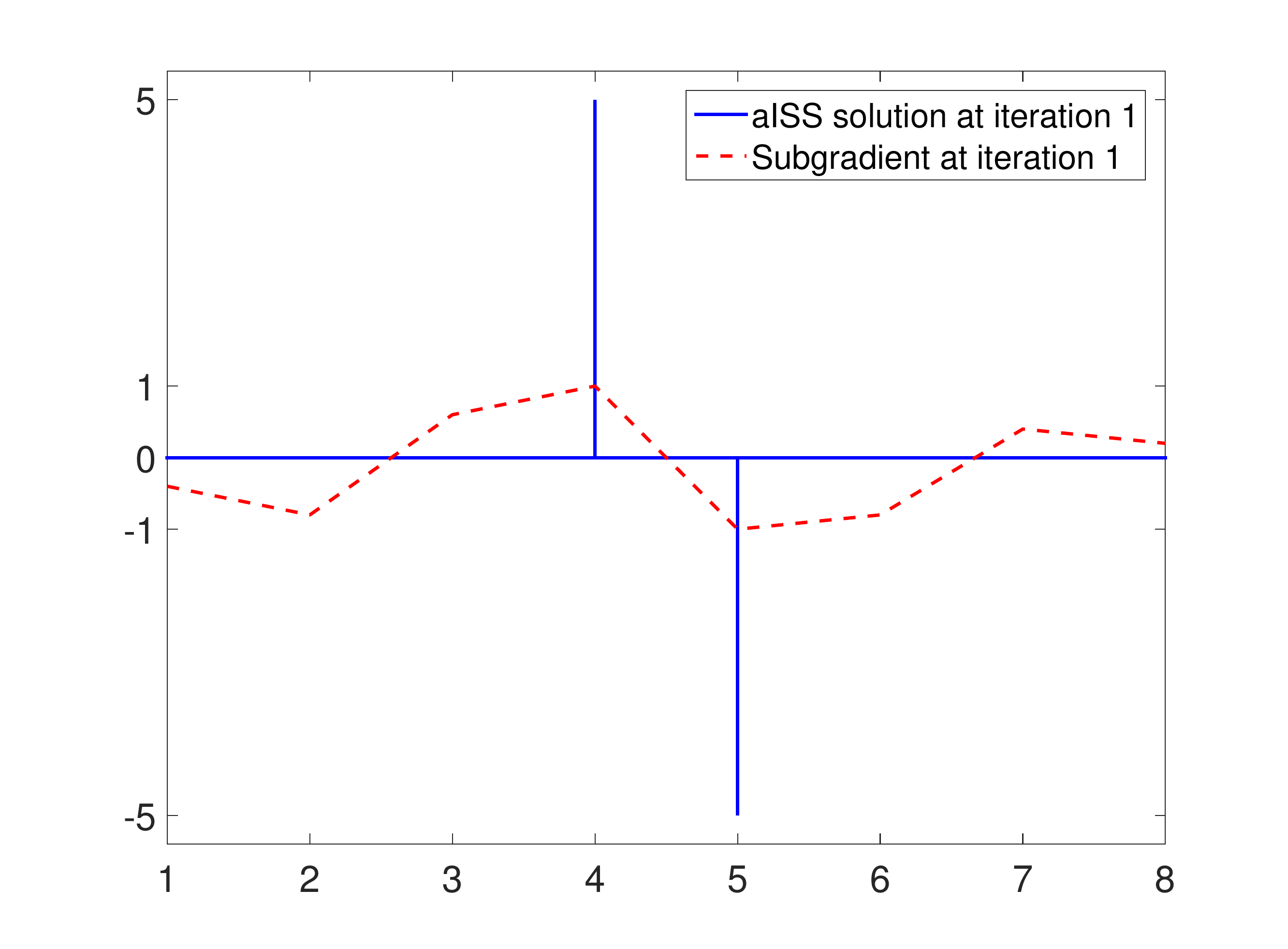}\label{subfig:aiss1}}
       \subfloat[Second aISS solution and subgradient.]{\includegraphics[width=0.49\textwidth]{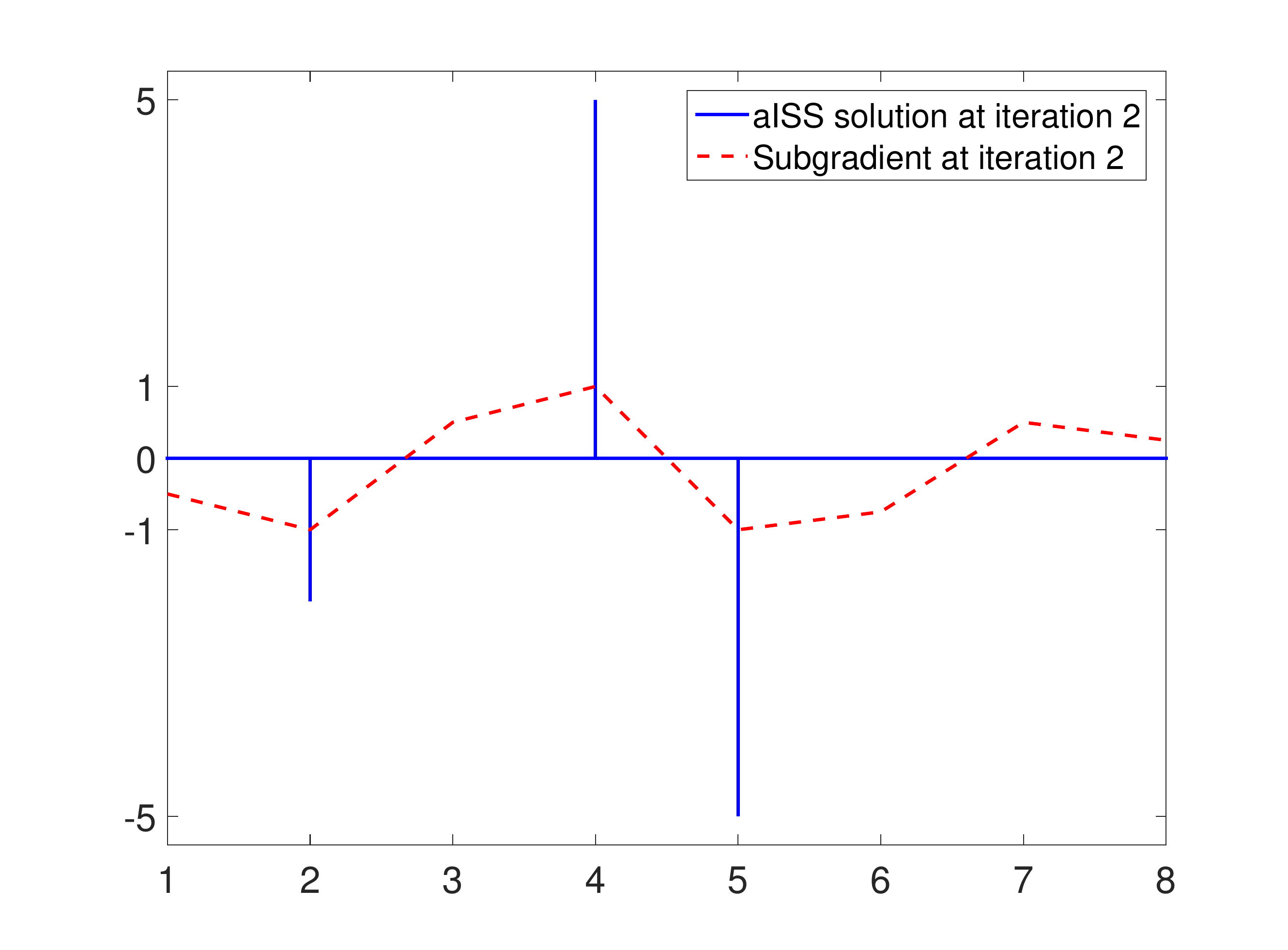}\label{subfig:aiss2}}\\
       \subfloat[Third and final aISS solution and subgradient.]{\includegraphics[width=0.49\textwidth]{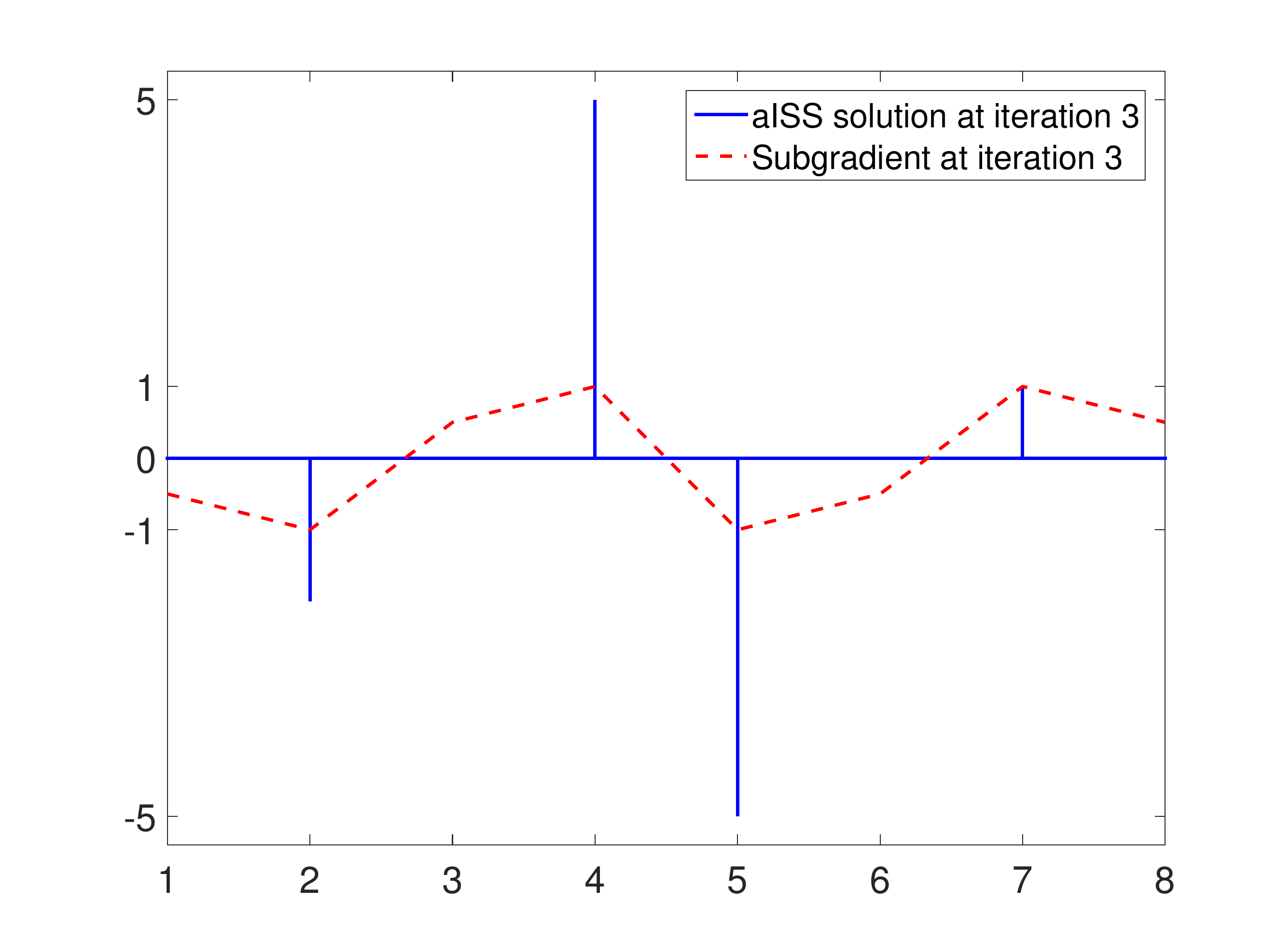}\label{subfig:aiss3}}
       \vspace{0.5cm}
\caption{The three iterations of the aISS algorithm until convergence for the example of Section \ref{sec: SUB0observance}. Figure \ref{subfig:aiss1} shows both the primal and the dual variable, i.e. $u(t_1)$, $p(t_1)$, at $t = t_1$, whereas Figure \ref{subfig:aiss2} and \ref{subfig:aiss3} show the same for $t_2$ and $t_3$, respectively. It is easily observed that the aISS algorithm has decomposed the signal $f$ into the original singular vectors with no loss of contrast, as stated by Theorem \ref{thm: ExactReconstruction}.}
\label{fig:positivesub0exm}
\end{figure}

\subsection{\eqref{eq: subdiff_cond} Violations}
\label{sec: SUB0violations}
Following Example \ref{example: l1_convolution} we also want to verify numerically that the absence of the \eqref{eq: subdiff_cond} condition in fact can lead to results different from the decomposition of the data into the original singular vectors. We use the operator $K$, the functional $J$, and the two singular vectors
\begin{align*}
u_{\lambda_1} = (0,0,1,0,0)^T \qquad \text{and} \qquad u_{\lambda_2} = \frac{1}{\sqrt{2}}(0,1,0,-1,0)^T,
\end{align*}
from example \ref{example: l1_convolution}. We know that the two singular vectors are $K$-orthogonal but violate the (\ref{eq: subdiff_cond}) condition. We can then for instance define $f = Ku_{\lambda_1} + \frac{3}{2\sqrt{2}} Ku_{\lambda_2}$ as potential input data that we want to decompose. Note that we have ensured $\lambda_1/\gamma_1 = 1 <  4/3 = \lambda_2/\gamma_2$. In order to compute a solution of the ISS flow \eqref{eq: ISSflow} we again use the aISS method.
The first non-trivial solution $u(t_1)$ of the aISS method reads as
\begin{align*}
u(t_1) = (0,5/4,0,0,0)^T.
\end{align*}
This solution is obviously as sparse as $u_{\lambda_1}$ and also a singular vector (with the same singular value after normalisation) but corresponds neither to $u_{\lambda_1}$ nor to $u_{\lambda_2}$. More importantly, it is a better fit to the data $f$ with respect to the $\ell^2$-norm, as we have $\| Ku(t_1) - f \|_{\ell^2(\R^m)}^2 = 9/16$ as opposed to $\| Ku_{\lambda_1} - f \|_{\ell^2(\R^m)}^2 = 9/8$. Hence, it is logical that the first ISS solution is not $u_{\lambda_1}$. The remaining ISS solutions are
\begin{align*}
u(t_2) = (0,1,1/2,0,0)^T \qquad \text{and} \qquad u(t_3) = (0,3/4,1,-3/4,0)^T,
\end{align*}
where $u(t_3)$ satisfies $u(t_3) = u_{\lambda_1} +  \frac{3}{2\sqrt{2}} u_{\lambda_2}$. What is interesting about this example is that the ISS flow \eqref{eq: ISSflow}, whilst converging to a $J$-minimial solution of $K^* Ku = K^* f$, most importantly has to minimise the residual in each step. This dominates the behaviour of keeping $J$ minimal at the same time, as we observe $J(u(t_1)) > J(u_{\lambda_1})$ but $\| Ku(t_1) - f \|_{\ell^2(\R^m)} < \| Ku_{\lambda_1} - f \|_{\ell^2(\R^m)}$. \\

We can obviously apply aISS method to compute solutions for compositions made from the singular vectors described in Example \ref{example: PartialSUB0Necessary} as well. We pick $\gamma_1 = 9$, $\gamma_2 = 8$, $\gamma_3 = 3\sqrt{2}$, $\gamma_4 = 2$, and $\gamma_5 = 1$ in order to have $\lambda_k/\gamma_k < \lambda_{k+1}/\gamma_{k+1}$ for $k=1,\dots,4$. In this case the first non-trivial solution $u(t_1)$ of the aISS method is given by
\begin{align*}
u(t_1) = (0,0,0,0,-11/2,0,0,0,0)^T.
\end{align*}
The recovered peak does not match any of the singular vectors $u_{\lambda_j}$, $j=1,\dots,5$. This shows that in order for the ISS flow to give a decomposition into the singular vectors of $f$ it is not enough that the full sum of the subgradients is in the subdifferential of $J$ at zero. We need the partial sums as well. The remaining solutions are
\begin{align*}
u(t_2) &= (0,0,5,0,-11/2,0,5,0,0)^T, \\
u(t_3) &= (0,0,5/4,15/2,-37/4,0,5,0,0)^T, \\
u(t_4) &= (0,0,0,12,-17,11,0,0,0)^T = \gamma_1u_{\lambda_1} + \gamma_2u_{\lambda_2} + \gamma_3u_{\lambda_3}, \\
u(t_5) &= (0,-2,0,12,-17,11,0,0,0)^T = \gamma_1u_{\lambda_1} + \gamma_2u_{\lambda_2} + \gamma_3u_{\lambda_3} + \gamma_4u_{\lambda_4}, \\
u(t_6) &= (0,-2,0,12,-17,11,0,-1,0)^T = \gamma_1u_{\lambda_1} + \gamma_2u_{\lambda_2} + \gamma_3u_{\lambda_3} + \gamma_4u_{\lambda_4} + \gamma_5u_{\lambda_5}.
\end{align*}
The solutions at steps two and three cannot be written as any linear combination of the five singular vectors $u_{\lambda_j}$, $j=1,\dots,5$. From step four the solutions can be written as linear combinations of the singular vectors with no loss of contrast. However, we are not able to separate $u_{\lambda_1}$, $u_{\lambda_2}$, and $u_{\lambda_3}$.

\section{Conclusions \& Outlook}
In this paper we have investigated the possibility of using the inverse scale space flow as a decomposition method for absolutely one-homogeneous regularisation functionals. We have formulated conditions under which the inverse scale space flow will give a decomposition of a linear combination of generalised singular vectors. Furthermore, we have investigated the behaviour of the flow for arbitrary data. We have given a characterisation of the first non-trivial solution and shown that this solution is a primal singular vector when the input data is what we have named a dual singular vector. \\ 

For the inverse scale space decomposition we require the generalised singular vectors representing the data to be orthogonal when mapped to the data space by the forward operator. This condition, however, is quite restrictive and in future research we aim to relax the condition. The idea is that the inverse scale space flow will then still give a decomposition into the singular vectors but with time-varying coefficients. Another subject of interest is to investigate whether the data can always be decomposed into a linear combination of singular vectors and a remainder that can be seen as noise satisfying special conditions such as (\ref{eq:noisecond}). Extending the decomposition results to data containing noise in the same manner as in \cite[Section 7.2]{Benning_Burger_2013} we then know that the inverse scale space flow will always give a decomposition into singular vectors. \\ 

Another direction for future research is the numerical computation of inverse scale space solutions as well as the numerical computation of generalised singular vectors. As mentioned earlier, the former is currently restricted to polyhedral regularisation functionals (cf. \cite{burger2013adaptive,moeller2013multiscale}). For the latter there has been made substantial progress in \cite{hein2010inverse,bresson2012convergence,benning2016learning,nossek2016flows}, but many open questions need yet to be addressed. How can we compute ground states as well as non-trivial singular vectors with larger singular values? How can we incorporate (linear) forward operators? How can we compute dual singular vectors in a robust fashion? These are only few of many open questions that need to be addressed in future research.

\section*{Acknowledgments} MFS acknowledges support from the Advanced Grant No. 291405 from the European Research Council as part of the project HD-Tomo. MB and CBS acknowledge support from the EPSRC Grant EP/M00483X/1, the Leverhulme Trust Grant 'Breaking the non-convexity barrier'. MB further acknowledges support from the Isaac Newton Trust and the Leverhulme Trust Early Career Fellowship Grant 'Learning from mistakes: a supervised feedback-loop for imaging applications'. CBS acknowledges support from the EPSRC Centre 'EP/N014588/1' and the Cantab Capital Institute for the Mathematics of Information. The authors thank Martin Burger, Antonin Chambolle, Guy Gilboa and Michael M\"{o}ller for fruitful discussions.

\section*{Data Statement}

{\color{red}The corresponding \copyright MATLAB-code will be put into a repository and made publicly available as the final version is submitted.}

\section*{References}
\bibliographystyle{iopart-num}
\bibliography{generalisedsvd}

\end{document}